 \documentclass{amsart} 

\usepackage{amsmath,amssymb,amsthm, mathrsfs, mathtools}
\usepackage{amscd,mathtools}
\usepackage{graphicx,epstopdf}
\usepackage{color,xcolor}
\usepackage{enumerate}
\usepackage{xspace}
\usepackage{tipa}
\usepackage[colorlinks=true]{hyperref}

\usepackage{stmaryrd}
\usepackage{tikz}
\usepackage{soul}
  
\usepackage[utf8]{inputenc}
\usepackage[T1]{fontenc}
\usepackage{imakeidx}
\makeindex[columns=3, title=Alphabetical Index, intoc]
\usepackage[shortlabels]{enumitem}
\usepackage{epsfig}
\usepackage{pinlabel}
\usepackage[all]{xy}

\usepackage{tikz}
\usetikzlibrary{shapes.geometric}
\usepackage{caption, subcaption}
\usepackage{tikz-cd}
\usepackage{tikz}
\usetikzlibrary{arrows}
\usetikzlibrary{decorations.pathreplacing}
\usetikzlibrary{intersections}
\usetikzlibrary{decorations.pathmorphing}

\usepackage{mathrsfs}
\usetikzlibrary{matrix}
\usetikzlibrary{positioning}
\DeclareMathAlphabet{\pazocal}{OMS}{zplm}{m}{n}
\tikzset{>=stealth}

  \newcommand{\calC}{\mathcal{C}}

  \newcommand{\res}{\mathcal{G}}

  \newcommand{\calL}{\mathcal{L}}
  
  \newcommand{\calN}{\mathcal{N}}

  \newcommand{\calZ}{\mathcal{Z}}

  \newcommand{\gothic}{\mathfrak}

  \newcommand{\go}{{\gothic o}}

  \newtheorem{theorem}{Theorem}[section]
  \newtheorem{proposition}[theorem]{Proposition}
  
  \newtheorem{lemma}[theorem]{Lemma}

  \newtheorem{introthm}{Theorem}

  \theoremstyle{definition}
  \newtheorem{definition}[theorem]{Definition}

  \newtheorem*{claim*}{Claim}

  \newtheorem*{question*}{Question}
  \newtheorem*{answer*}{Answer}
  \newtheorem*{application*}{Application}

  \theoremstyle{remark}
  \newtheorem{remark}[theorem]{Remark}
  \newtheorem*{remark*}{Remark}
  
  \newcommand{\propref}[1]{Proposition~\ref{#1}}

  \newcommand{\defref}[1]{Definition~\ref{#1}}
  
  \newcommand{\eqnref}[1]{Equation~\eqref{#1}}

\DeclareMathOperator{\minset}{minset}

  \DeclareMathOperator{\Len}{Len}

  \DeclareMathOperator{\cent}{ctr}

  \newcommand{\pka}{\partial_{\kappa}}

\DeclareMathOperator{\diam}{diam}

  \newcommand{\param}{{\mathchoice{\mkern1mu\mbox{\raise2.2pt\hbox{$
  \centerdot$}}
  \mkern1mu}{\mkern1mu\mbox{\raise2.2pt\hbox{$\centerdot$}}\mkern1mu}{
  \mkern1.5mu\centerdot\mkern1.5mu}{\mkern1.5mu\centerdot\mkern1.5mu}}}

\DeclarePairedDelimiterX{\norm}[1]{\lvert}{\rvert}{#1}
\DeclarePairedDelimiterX{\Norm}[1]{\lVert}{\rVert}{#1}

  \renewcommand{\setminus}{{\smallsetminus}}

  \newcommand{\from}{\colon\thinspace}

\newcommand{\CAT}{\ensuremath{\operatorname{CAT}(0)}\xspace}

\usepackage{ifthen}

\newcommand{\showcomments}{yes}

\newsavebox{\commentbox}
{\ifthenelse{\equal{\showcomments}{yes}}%
{\footnotemark
    \begin{lrbox}{\commentbox}
    \begin{minipage}[t]{1.25in}\raggedright\sffamily\upshape\tiny
    \footnotemark[\arabic{footnote}]}%
{\begin{lrbox}{\commentbox}}}%
{\ifthenelse{\equal{\showcomments}{yes}}%
{\end{minipage}\end{lrbox}\marginpar{\usebox{\commentbox}}}%
{\end{lrbox}}}

\newcounter{acomments}

\newcounter{ccomments}

\begin{document}

\title[Sublinearly Morse Boundary of \CAT admissible groups]{Sublinearly Morse Boundary of \CAT admissible groups}

\author{Hoang Thanh Nguyen}
\address{Department of Mathematics\\
The University of Danang - University of Science and
Education\\
 459 Ton Duc Thang, Da Nang, Vietnam}
\email{nthoang.math@gmail.com}

\author   {Yulan Qing}
\address{Shanghai Center for Mathematical Sciences, Fudan University, Songhu Road 2005, Shanghai, China}
\email{yulan.qing@gmail.com}
%

  
\maketitle

\begin{abstract}
We show that if $G$ is an admissible group acting geometrically on a \CAT space $X$, then $G$ is a hierarchically hyperbolic space and its $\kappa$-Morse boundary $(\pka G, \nu)$ is a model for the Poisson boundary of $(G, \mu)$ where $\nu$ is the hitting measure associated to the random walk driven by $\mu$( with mild assumptions of $\mu$).
\end{abstract}

\section{Introduction}
Sublinearly Morse boundaries are recently constructed for all finitely generated groups \cite{QRT19, QRT20}.  It is a metrizable topological space that is a group invariant. Similar Gromov boundary of hyperbolic spaces, sublinearly Morse boundaries are particularly illuminating in revealing features of groups that contains hyperbolic-like features \cite{IZ21, MQZ20, QZ21, Zal20}. One of these feature centers around asymptotic behavior of random walk on the associated groups. In several important classes of groups, such right-angled Artin groups, relative hyperbolic groups, and the mapping class groups of surfaces of finite type, an appropriately chosen sublinear function yields a sublinearly Morse boundary that serves as a topological model for the Poisson boundaries(with mild assumptions) of the group. 

In this paper we continue to prove the connection between a geometric boundary and random walk boundary for a new class of groups: \CAT admissible groups.

\CAT admissible groups is a particular type of graph of groups that generalizes graph manifolds. A compact, orientable, irreducible 3-manifold with empty or tori boundary is called a graph manifold if it is obtained by gluing finitely many Seifert manifolds with orientable hyperbolic base orbifolds where the gluing maps between the Seifert components do not identify (unoriented) Seifert fibers up to homotopy. \CAT admissible group was first introduced by Croke-Kleiner \cite{CK02} as key counterexamples to the nice properties of Gromov boundary when the hyperbolicity assumption is dropped. This property is further studied in \cite{qing1}.

Roughly speaking, a \CAT admissible group models the JSJ structure of nonpositively curved graph manifolds where the Seifert fibration is replaced by the following central extension of a general hyperbolic group:
\begin{equation}\label{centralExtEQ}
1\to Z(G_v)=\mathbb Z\to G_v\to H_v\to 1    
\end{equation}

In this paper, we also refer to \CAT admissible groups as CKA groups (CKA stands for Croke-Kleiner admissible). This class of groups serves as a simple algebraic means to produce interesting groups from an arbitrary finite collection of $n$ \CAT hyperbolic groups.

\medskip
\noindent \textbf{A basic example}. Let $H_1$ and $H_2$ be two torsion-free hyperbolic groups which act geometrically on $\CAT$ spaces $X_1$ and $X_2$ respectively. Then $G_i = H_i \times \langle t_i \rangle$ (with $i =1,2$)  acts geometrically on the $\CAT$ space $Y_ i = X_{i} \times \mathbb R$. Any primitive hyperbolic element $h_i$ in $H_i$ gives a totally geodesic torus $T_i$ in the quotient space $Y_{i}/G_i$ with basis $([h_i], [t_i])$. We rescale $Y_ i$ so that the lengths of $[h_1]$ and of $[t_2]$ are the same, and the lengths of $[t_1]$ and $[h_2]$ are the same. Let $f \colon T_1 \to T_2$ be an isometry that flips the $\mathbb Z$-axis with the axis of an loxodromic element in $H_i$. Let $M$ be the space obtained by gluing $Y_1$ to $Y_2$ by the isometry $f$. One can show that there exists a metric on $M$ such that with respect to this metric, $M$ is a locally $\CAT$ space (see e.g. \cite[Proposition 11.6]{BH99}). By the Cartan-Hadamard Theorem, then the universal cover $\widetilde M$ with the induced length metric from $M$ is a \CAT space, and thus $\pi_1(M)$ is a \CAT admissible group and $\pi_1(M) \curvearrowright \widetilde M$ is a \CAT action.

\begin{introthm}
\label{introthm:1}

Let $G$ be a \CAT admissible group. The following holds. 
\begin{enumerate}
\item  $G$ is a hierarchically hyperbolic space with respect to a word metric. 
    \item  Let $\mu$ be a finitely supported, non-elementary probability measure on $G$. Then for $\kappa(t) = \log^p(t)$, where $p$ is the complexity of the HHS. Then almost every sample path $(w_n)$ converges to a point in $\pka G$; and the $\kappa$-Morse boundary $(\pka G, \nu)$ is a model for the Poisson boundary of $(G, \mu)$where $\nu$ is the hitting measure associated to the random walk driven by $\mu$.
\end{enumerate}
\end{introthm}

The key step to show Theorem~\ref{introthm:1} involves proving that \CAT admissible groups acts geometrically on hierarchically hyperbolic spaces. Hierarchically hyperbolic spaces are introduced axiomatically, modeling the Masur–Minsky deconstruction of mapping class groups \cite{MM00}. A hierarchically hyperbolic space (HHS) consists of a metric space $X$ together with a partially ordered set of $\delta$-hyperbolic spaces. There are projections maps from $X$ to each of the hyperbolic spaces and from the hyperbolic spaces to each other.  In some sense, this set of hyperbolic spaces form a tree-like metric that can be used to measure distances in $X$. Many interesting groups are quasi-isometric to hierarchically hyperbolic spaces, examples of classes of groups include but are not limited to:
\begin{itemize}
    \item hyperbolic groups and relative hyperbolic groups 
    \item Mapping class groups of orientable surfaces of finite type.
     \item Fundamental groups of compact three-manifolds with no Nil or Sol in their prime decomposition.
      \item  Groups that act properly and cocompactly on a special \CAT cube complex, and more
generally any cubical group which admits a factor system. 
\end{itemize}

It is established in \cite{BHS19} that all graph manifolds are a hierarchically hyperbolic space. To show this, we have Kapovich-Leeb \cite{KL98}   to show that all graph manifolds are quasi-isometric to flip graph manifolds, and the combination theorem of \cite{BHS19} shows that the fundamental group of a flip graph manifold is HHS. When we lose the manifold structure and instead we look at graph of groups and in the setting of \CAT admissible groups, it is unknown whether all fundamental group of graph of hyperbolic-by-Z groups are quasi-isometric to the ones whose gluing map is a flip isometry(here flip roughly means the same as in the basic example). In Theorem~\ref{introthm:1}, we drop both the assumption of flip and also the assumption of the space being a manifold. 

Lastly, an intermediate step in the proof of Theorem~\ref{introthm:1} depends on the following fact:

\begin{introthm}\label{T:PB-intro}
Let $\mu$ be a non-elementary, finitely supported measure on $G$ and assume that $G$ acts geometrically on an hierarchically hyperbolic space. Then for $\kappa(r) := \log^p(r)$, where $p$ is the complexity of the hierarchy, the $\kappa$-Morse boundary is a topological model for the Poisson boundary of $(G, \mu)$. 
\end{introthm}

The claim is proven for all mapping class groups and is remarked in \cite{QRT20}. We include the proof for hierarchically hyperbolic spaces in the appendix of this paper. 

\subsection{Acknowledgments}
 Hoang Thanh Nguyen is partially supported by Project ICRTM04\_2021.07 of the International Centre for Research and Postgraduate Training in
Mathematics, VietNam. We thank Mark Hagen and Kasra Rafi for useful conversations.

\section{Preliminaries}
In this section, we review definitions of $\kappa$-projection, $\kappa$-Morse boundary in Section~\ref{pre:2.1}, random walk on groups and Posisson boundaries in Section~\ref{pre:2.2}. These two sections (Section~\ref{pre:2.1} and Section~\ref{pre:2.2}) are only used in Section~\ref{appendix}, so the reader could skip these two sections until Section~\ref{appendix} . In Section~\ref{pre:2.3}, we recall the definition of hierachically hyperbolic space (HHS) given in \cite{BHS19}. We then review backgrounds on CAT(0) admissible groups in Section~\ref{admissible}.
\subsection{$\kappa$-projection, $\kappa$-Morse boundary}
\label{pre:2.1}
We refer to and follow the construction in \cite{QRT20} as a complete reference for $\kappa$-Morse boundaries of proper geodesic metric spaces. Here we briefly recall the construction.
\begin{definition} \label{weakprojection}
Let $(X, d_X)$ be a proper geodesic metric space and $Z \subseteq X$ a closed subset, and let 
$\kappa$ be a concave sublinear function. A map $\pi_{Z} \from X \to \mathcal{P}(Z)$ is a $\kappa$-\emph{projection}  if there exist constants $D_{1}, D_{2}$, depending only on $Z$ and $\kappa$, such that for any points $x \in X$ and $z \in Z$, 
\[
\diam_X(\{z \} \cup \pi_{Z}(x)) \leq D_{1} \cdot d_X(x, z) + D_{2} \cdot \kappa(x).
\]
\end{definition}

A $\kappa$-projection differs from a nearest point projection by a uniform multiplicative error and a sublinear additive error\cite{QRT20}.

\begin{lemma} \cite{QRT20}
Given a closed set $Z$, we have for any $x \in X$
 $$\diam_X(\{x \} \cup  \pi_{Z}(x)) \leq (D_{1} + 1) \cdot d_X(x, Z) + D_{2}\cdot \kappa(x).$$
\end{lemma}

\begin{definition} \label{D:k-morse}
Let $Z \subseteq X$ be a closed set, and let $\kappa$ be a concave sublinear function. 
We say that $Z$ is \emph{$\kappa$-Morse} if there exists a proper function 
$m_Z : \mathbb{R}^2 \to \mathbb{R}$ such that for any sublinear function $\kappa'$ 
and for any $r > 0$, there exists $R$ such that for any $(q, Q)$-quasi-geodesic ray $\beta$
with $m_Z(q, Q)$ small compared to $r$, if 
$$d_X(\beta_R, Z) \leq \kappa'(R)
\qquad\text{then}\qquad
\beta|_r \subset \calN_\kappa \big(Z, m_Z(q, Q)\big)$$
The function $m_Z$ will be called a $\emph{Morse gauge}$ of $Z$. 
\end{definition}

\begin{definition}
Let $X$ be a proper geodesic space with a base-point $\go$. Two $\kappa$-Morse quasi-geodesic rays $\alpha$ and $\beta$ are equivalent, written as $\alpha \sim \beta$, if they diverge sublinearly. The set of all equivalence classes of $\kappa$-Morse geodesic quasi-geodesic rays, together with the \emph{fellow traveling quasi-geodesic} topology \cite{QRT20}.
\end{definition}

\begin{theorem}\cite{QRT20}
Let $X$ be a proper geodesic metric space and let $\kappa$ be a sublinear function. The associated $\kappa$-Morse boundary $\pka X$ is a metrizable topological space and $\pka X$ is quasi-isometrically invariant. 
\end{theorem}

\subsection{Random walk and Poisson boundaries}
\label{pre:2.2}
We now  review random walk on groups per the study of this project. Let $G$ be a locally compact, second countable group, with left Haar measure $m$, and let $\mu$ be a Borel probability measure on $G$, 
which we assume to $\emph{spread-out}$, i.e. such that there exists $n$ for which $\mu^n$ is not singular w.r.t. $m$. 
Given $\mu$, we consider the \emph{step space} $(G^\mathbb{N}, \mu^\mathbb{N})$, whose elements we denote as $(g_n)$. 
The \emph{random walk driven by }$\mu$ is the $G$-valued stochastic process $(w_n)$, where for each $n$ we define the product
$$w_n := g_1 g_2 \dots g_n.$$
We denote as $(\Omega, \mathbb{P})$ the \emph{path space}, i.e. the space of sequences $(w_n)$, where $\mathbb{P}$ is the measure induced by pushing forward the measure $\mu^\mathbb{N}$ from the step space. Elements of $\Omega$ are called \emph{sample paths} and will be also denoted as $\omega$. Let $T : \Omega \to \Omega$ be the left shift on the path space, and let $(B, \mathcal{A})$ be a measurable space on which $G$ acts by measurable isomorphisms;  a measure $\nu$ on $B$ is $\mu$-\emph{stationary} if $\nu = \int_G g_\star \nu \ d\mu(g)$, and in that case the pair $(B, \nu)$ is called a $(G, \mu)$-\emph{space}. 
Recall that a \emph{$\mu$-boundary} is a measurable $(G, \mu)$-space $(B,\nu)$ such that there exists 
a $T$-invariant, measurable map $\textbf{bnd} : (\Omega, \mathbb{P}) \to (B, \nu)$, called the \emph{boundary map}. 

Moreover, a function $f: G \to \mathbb{R}$ is $\mu$-\emph{harmonic} if $f(g) = \int_G f(gh) \ d\mu(h)$ for any $g \in G$. 
We denote by $H^\infty(G, \mu)$ the space of bounded, $\mu$-harmonic functions. 
One says a $\mu$-boundary is the \emph{Poisson boundary} of $(G, \mu)$ if the map 
\[
\Phi : H^\infty(G, \mu) \to L^\infty(B, \nu)
\]
given by $\Phi(f)(g) := \int_B f \ dg_\star \nu$ is a bijection. 
The Poisson boundary $(B, \nu)$ is the maximal $\mu$-boundary, in the sense that for any other $\mu$-boundary $(B', \nu')$ there exists a $G$-equivariant, 
measurable map $p: (B, \nu) \to (B', \nu')$.  For more details, see \cite{Kai00}. 

%
%

\subsection*{The $\kappa$-Morse boundary realizes the Poisson boundary}
The following result is due to Ilya Gekhtman and \cite{QRT20}.

\begin{theorem} \label{T:poiss-general}
Let $G$ be a finitely generated group, and let $(X, d_X)$ be a Cayley graph of $G$. 
Let $\mu$ be a probability measure on $G$ with finite first moment with respect to $d_X$, such that the semigroup generated 
by the support of $\mu$ is a non-amenable group. 
Let $\kappa$ be a concave sublinear function, and suppose that
for almost every sample path $\omega = (w_n)$, there exists a $\kappa$-Morse geodesic ray $\gamma_\omega$ such that 
\begin{equation} \label{E:sub-track}
\lim_{n \to \infty} \frac{d_X(w_n, \gamma_\omega)}{n} = 0.
\end{equation}
Then almost every sample path converges to a point in $\pka X$, and moreover the space $(\pka X, \nu)$, where $\nu$ is the hitting measure for the random walk, is a model for the Poisson boundary of $(G, \mu)$. 
\end{theorem}

\subsection{Hierarchically hyperbolic spaces}
\label{pre:2.3}
We first recall the definition of a hierarchically hyperbolic space (HHS) as given in \cite{BHS19}.

\begin{definition}[Hierarchically hyperbolic spaces]
\label{defn:HHS}
A $q$--quasigeodesic space $(X,d)$ (i.e, there exists a constant $q$ such that any two points in $X$ is joining by a $(q,q)$--quasigeodesic segment) is a hierarchically hyperbolic space if there exists $\delta \ge 0$, an index set $\Lambda$, a collection of $\delta$--hyperbolic spaces $\{\calC W: W \in \Lambda \}$ such that the following conditions are satisfied:

    {\textbf{ Projections:}} There exists a uniform constant $\xi \ge 0$, there is a set \[\{ \pi_{W} \colon X \to 2^{\calC W},  W \in \Lambda \} \] of projections such that $\diam_{\calC W}(\pi_{W}(x)) \le \xi$ for all $x \in \calC W$. Moreover, there exists $K$ such that for all $W \in \Lambda$, $\pi_{W}$ is coarsely  Lipschitz and the image $\pi_1(W)(X)$ (which is defined to be the union of $\pi_{W}(x)$) is a $K$--quasiconvex set in $\calC W$.
    
    {\textbf { Nesting:}} $\Lambda$ is equipped with a partial order $\sqsubseteq$ such that either $\Lambda = \emptyset$ or $\Lambda$ contains a unique $\sqsubseteq$--maximal element. In this paper the maximal element is always denoted $S$. When $V \sqsubseteq W$ we say $V$ is nested in $W$. For each $W \in \Lambda$, define $\Lambda_{W} = \{V \in \Lambda : V \sqsubseteq W \}$.
     If $V, W \in \Lambda$ such that $V$ is properly nested in $W$ then there is a specific subset $\rho_{W}^{V} \subset \calC W$ such that $\diam (\rho_{W}^{V}) \le \xi$.
    Also, there is a projection $\rho_{V}^{W}$ from $\calC W$ to subsets of $\calC V$.
    
    {\textbf {Orthogonality:}}
    $\Lambda$ has a symmetric and anti-reflexive relation called {\it orthogonality}: (we use notation $ \perp$). Also, whenever $V \sqsubseteq W$ and $W \perp U$ then we require that $V \perp U$.
        We require that for each $S \in \Lambda$, for each $U \in \Lambda_{S}$ such that the set $\{ V \in \Lambda_{S}: V \perp U \} \neq \emptyset$ then there exists $W \in  \Lambda_{S}$, $W \neq S$ so that whenever $V \perp U$, $V \sqsubseteq S$ then we have $V \sqsubseteq W$. Finally, if $V \perp W$, then $V, W$ are not $\sqsubseteq$--comparable.

     {\textbf {Transversality and consistency:}}
If $V$ and $W$ are not orthogonal and neither is nested in the other, then we say that $V$ and $W$ are {\it transverse }, denoted $
V \pitchfork W
$, then there are subsets $\rho^V_W \subset \calC W$ and $P^{W}_{V} \subset \calC V$ such that $\diam(\rho^V_W) \le \xi$, $\diam(P^{W}_{V}) \le \xi$ and satisfying:

   For any $x \in X$ we have
 \begin{equation}
    \tag{$\clubsuit$}
    \label{eqn:1}
        \min \{ d_{W} (\pi_{W}(x)), P^V_W) \, , d_{V}(\pi_{V}(x), P^{W}_{V}) \} \le \kappa_{0}
    \end{equation}

 For each $V, W \in \Lambda$ such that $V \sqsubseteq W$, and for each $x \in X$ we have:
\begin{equation}
\tag{$\spadesuit$}
\label{eqn:2}
    \min \{d_{W}(\pi_{W}(x), P^{V}_{W}) \, , \diam ( \pi_{V}(x) \cup \rho^{W}_{V}(\pi_{W}(x))) \} \le \kappa_{0}
\end{equation}

\begin{equation}
    \tag{$\diamondsuit$}
    \label{eqn:3}
    \text{If} \qquad  U \sqsubseteq V,\, \qquad \text{then} \qquad \,  d_{W}(\mathcal{P}^{U}_{W}, \mathcal{P}^V_W) \le \kappa_{0}
\end{equation}
whenever $W \in \Lambda$ satisfies either $V$ is properly nested in $W$ 

    {\textbf {Finite complexity: }}
  The {\it complexity}  of $X$  is an integer $n \ge 0$ such that any set of pairwise $\sqsubseteq$--compatible elements  has  cardinality at most $n$.

    {\textbf {Large links:} } 
 There exists $\lambda \ge 1$ and $E \ge \xi, \kappa_{0}$ such that the following holds.    Let $W \in \Lambda$ and let $x, x' \in X$. Let $N = \lambda d_{w} (\pi_{W}(x), \pi_{W}(x')) + \lambda$. Then there exists $A_1, A_2, \cdots, A_{\lfloor N \rfloor} \in \Lambda_{W} \backslash \{W\}$ such that for all $A \in \Lambda_{W} \backslash \{W\}$ then either $A = T_i$ for some $i$, or $d_{A} (\pi_{A}(x), \pi_{A}(x')) < E$. Also, $d_{W}(\pi_{W}(x), P^{A_{i}}_{W}) \le N$ for each $i$.

    {\textbf{ Bounded geodesic image: }}
    There exists  $E >0$  such that for all $W \in \Lambda$, all $V \in \Lambda_{W} \backslash \{W\}$, and all geodesic $\gamma$ of $\calC W$, either 
    \begin{equation}
    \tag{$\oplus$}
    \label{eqn:4}
          \diam_{\calC V}(\rho^{W}_{V}(\gamma)) \le E \,\,\text{or}\,\, \gamma \cap \mathcal{N}_{E}(\rho^V_W) \neq \emptyset
    \end{equation}

    {\textbf {Partial Realization:}}
  There exists a constant $\alpha$ with the following property. Let $\{V_j\}$ be a family of pairwise orthogonal elements of $\Lambda$, and let $p_{j} \in \pi_{V_{j}}(X)$. Then there exists an element $x \in X$ such that
\begin{itemize}
    \item $d_{V_{j}}(x, p_j) \le \alpha$ for all $j$
    \item For each $j$, and each $v \in \Lambda$ with $V_{j} \sqsubseteq V$, we have $d_{V}(x, \rho^{V_{j}}_{V}) \le \alpha$ and
    \item If $W$ and $V_j$ are transverse for some $j$ then $d_{W}(x, p^{V_{j}}_{W}) \le \alpha$.
\end{itemize}

     {\textbf {Uniqueness:}}
    For each $k \ge 0$, there exists $\Theta = \Theta(k)$ such that for any $x, y$ in $X$, $d_{X}(x, y) \ge \Theta$ then there exists $V \in \Lambda$ such that $d_{V}(x,y) \ge k$.

\end{definition}

\subsection{\CAT admissible groups}
\label{admissible}
\emph{Admissible groups} firstly introduced in \cite{CK02} are a particular class of graph of groups that includes fundamental groups of $3$--dimensional graph manifolds. In this section, we review  admissible groups and their properties that will used throughout in this paper.

We often consider oriented edges from $e_-$ to $e_+$ and denote by $e=[e_-,e_+]$. 
Then $\bar e=[e_+,e_-]$ denotes the oriented edge with reversed orientation. Denote by $\mathcal G^0$ the set of vertices and by $\mathcal G^1$ the set of all oriented edges. 

\begin{definition}
\label{defn:admissible}
A graph of group $\mathcal{G}$ is \emph{admissible} if
\begin{enumerate}
    \item $\mathcal{G}$ is a finite graph with at least one edge.
    \item Each vertex group ${ G}_v$ has center $Z({ G}_v) \cong \mathbb Z$, ${ Q}_v \colon = { G}_{v} / Z({ G}_v)$ is a non-elementary hyperbolic group, and every edge subgroup ${ G}_{e}$ is isomorphic to $\mathbb Z^2$.
    \item Let $e_1$ and $e_2$ be distinct directed edges entering a vertex $v$, and for $i = 1,2$, let $K_i \subset { G}_v$ be the image of the edge homomorphism ${G}_{e_i} \to {G}_v$. Then for every $g \in { G}_v$, $gK_{1}g^{-1}$ is not commensurable with $K_2$, and for every $g \in  G_v \setminus K_i$, $gK_ig^{-1}$ is not commensurable with $K_i$. 
    \item For every edge group ${ G}_e$, if $\alpha_i \colon { G}_{e} \to { G}_{v_i}$ are the edge monomorphism, then the subgroup generated by $\alpha_{1}^{-1}(Z({ G}_{v_1}))$ and $\alpha_{2}^{-1}(Z({ G}_{v_1}))$ has finite index in ${ G}_e$.
\end{enumerate}
A group $G$ is \emph{admissible} if it is the fundamental group of an admissible graph of groups.

\end{definition}

\begin{definition}
\label{defn:admissibleaction}
We say that an admissible group $G$ is \CAT if it acts properly discontinuously, cocompactly and by isometries on a complete proper \CAT space. For the remainder of this paper, the pair $(G, X)$ will always denote a \CAT admissible group with the associated \CAT space. For examples of admissible groups that are not \CAT, see \cite{KL96}.
\end{definition}

Let $G$ be a \CAT admissible group, and let $G \curvearrowright T$ be the action of $G$ on the associated Bass-Serre tree $T$ of the graph of group $\mathcal{G}$ (we refer the reader to Section~2.5 in \cite{CK02} for a brief discussion). Let $V(T)$ and $E(T)$ be the vertex and edge sets of $T$.  For each $\sigma \in V(T) \cup E(T)$, we let $G_{\sigma} \le G$ be the stabilizer of $\sigma$. 

In the rest of this section, we review facts from Section~3.2 in \cite{CK02} that will be used thoroughly this paper. We refer the reader to \cite{CK02} for further explanation. Given $(G,X)$ pair, 
\begin{enumerate}
    \item for every vertex $v\in V(T),$ the set $B_{v} :=  \cap_{g \in Z(G_v)} \minset(g)$ splits as metric product $H_v \times \mathbb R$ where $Z(G_v)$ acts by translation on the $\mathbb R$--factor and  the quotient $G_v/Z(G_v)$ acts geometrically on the Hadamard space $H_v$. 
    
    \item for every edge $e \in E(T)$, the minimal set $B_{e} :=  \cap_{g \in G_e} \minset(g)$ splits as $\overline {B_e} \times \mathbb R^2\subset B_v$ where $\overline {B_e}$ is a compact \CAT space and $G_e=\mathbb Z^2$ acts co-compactly on the Euclidean plane $\mathbb R^2$.
\end{enumerate}
We note that the assignments $v \to B_v$ and $e \to B_e$ are $G$--equivariant with respect to the natural $G$ actions in the sense that $gB_{v}  = B_{gv}$ and $gB_{e} = B_{ge}$ for any $g \in G$.

{\textbf{Strips in X}:} We reviews strips in $X$ that was discussed in the second paragraph in Section~4.1 in \cite{CK02}). We first choose, in a $G$--equivariant, a plane $F_e \subset B_{e}$ for each edge $e$ in $T$ (i.e, $gF_{e} = F_{ge}$ for each $g \in G$). We will call $F_{e}$ is a {\textit {boundary plane}} of $B_{e_{\pm}}$.  
For every pair of adjacent edges $e_1$, $e_2$. we choose, again equivariantly, a shortest geodesic $\alpha_{e_1 e_2}$ from $F_{e_1}$ to $F_{e_2}$ (equivariantly here we mean $\alpha_{(ge)(ge')} = g \alpha_{ee'}$ for adjacent edges $e$, $e'$ ). By the convexity of $B_v = H_v \times \mathbb R$, $v:= e_{1} \cap e_{2}$, this geodesic $\alpha_{e_1e_2}$ determines a Euclidean strip $\mathcal{S}_{e_1e_2} : = \gamma_{e_1e_2} \times \mathbb R$ (possibly of width zero) for some geodesic segment $\gamma_{e_1e_2} \subset H_v$.

\begin{remark}
\label{rem:indexfunction}
There exists a $G$--equivariant coarse $L$--Lipschitz map  $\pi \colon X \to V(T)$  such that  $x \in X_{\rho(x)}$ for all $x \in X$. The map $\pi$ is called a {\textit{ index map}}. We refer the reader to Section~3.3 in \cite{CK02} for  the existence of such a map $\rho$.
 \end{remark}

 \begin{lemma} \cite[Lemma 2.9]{CK02}
 \label{lemma:epsilon}
 There exists a constant $\epsilon > 0$ such that the following hold. Let $F_e$ and $F_e'$ be two boundary planes in a vertex space $B_v$. Let $\mathcal{S}_{e,e'} = \gamma_{e,e'} 
 \times \mathbb{R}$ be the strip in $B_v$. Let $\ell: = F_{e} \cap H_v$ and $\ell' = F_{e'} \cap H_v$. Then the $\ell \cup \gamma_{e,e'} \cup \ell'$ is a $\epsilon$--quasiconvex in the \CAT hyperbolic space $H_v$.
 \end{lemma}

\begin{definition}[Boundary lines]
\label{defn:boundaryline}
 Let $\{F_e \}$ be the collection of boundary planes of the block $B_v$. We note that the intersection of a boundary plane $F_e$ of $B_v$ with the hyperbolic space $H_v$ is a line. We define the collection of lines $\mathbb{L}_v$, which we call {\textit{boundary lines}}, in the hyperbolic space $H_v$ as the following:
$$
\mathbb{L}_{v} = \{ \ell_e = F_{e} \cap H_v \, | \, e_{-} = v \}
$$
\end{definition}

\begin{remark}
\label{rem:rdistance}
There exists a constant $r >0$ such that for any vertex $v \in T$, for any $x \in H_v$, there exists  a boundary line $\ell \in \mathbb{L}_v$ such that $d(x, \ell) \le r$.
\end{remark}

The following lemmas are needed in Section~\ref{main} to check the HHS Axioms.

\begin{lemma}
\label{numberlines}
 $Q_v$ is hyperbolic relative to $\mathbb P:=\{\bar K_e: e_-=v, e\in \mathcal G^1\}$ where $\bar{K}_e$ is the projection of $G_e$ under the quotient $G_v \to Q_v$.
\end{lemma}

\begin{proof}
Let $K_e$ be the image  of an edge  group $G_e$ into $G_v$ and $\bar K_e$ be its projection in $Q_v$ under $G_v\to Q_v=G_v/Z(G_v)$. Then $\mathbb P:=\{\bar K_e: e_-=v, e\in \mathcal G^1\}$ is an almost malnormal collection of virtually cyclic subgroups in $H_v$. Indeed, since $Z(G_v)\subset K_e \cong \mathbb Z^2$, we have $\bar K_e=K_e/Z(G_v)$ is virtually cyclic. The almost malnormality follows from non-commensurability of $K_e$ in $G_v$. To see this,  assume that $\bar K_e\cap h\bar K_{e'}h^{-1}$ contains an infinite order element by the hyperbolicity of $Q_v$. If  $g\in G_v$ is sent to $h$, then $K_e\cap gK_{e'}g^{-1}$ is sent to $\bar K_e\cap h\bar K_{e'}h^{-1}$.    Thus, $K_e\cap gK_{e'}g^{-1}$ contains an abelian group of rank 2. The non-commensurability of $K_e$ in $G_v$ implies that  $e=e'$ and $g\in K_e$. This shows that $\mathbb P$ is almost malnormal.
It is well-known that a hyperbolic group is hyperbolic relative to any almost malnormal collection of quasi-convex subgroups (\cite{Bow12}). Thus  $Q_v$ is hyperbolic relative to $\mathbb P$. 
\end{proof}

\section{\CAT admissible groups are HHS}
\label{main}

In this section, we are going to prove the following proposition (that is also part (1) in Theorem~\ref{introthm:1}).

\begin{proposition}
\label{prop:part1}
Let $G$ be a \CAT admissible group. Then $G$ is a hierarchically hyperbolic space with respect to a word metric.
\end{proposition}

First at all, we will construct an index set $\Lambda$ and a set $\{\calC W: W \in \Lambda$ \} of $\delta$--hyperbolic spaces that satisfies HHS's axioms (see Definition~\ref{defn:HHS}).
We note that if metric spaces $\widetilde{X}$ and $\widetilde{Y}$ are quasi-isometric, then $\widetilde X$ is a hierarchically hyperbolic space whenever $\widetilde Y$ is a hierarchically hyperbolic space (see Proposition~1.10 in \cite{BHS19}). Hence to see that $G$ is a hierarchically hyperbolic space with respect to word metrics, we only need to work on its appropriate finite index subgroups.

Define $V(T):=\mathcal V_1\cup\mathcal V_2$ where $\mathcal V_i$ consists of vertices in $T$ with pairwise even distances.  
Passing to a finite index subgroup if necessary, we could assume that $G$ preserves each $\mathcal{V}_i$ (see Lemma~4.6 in \cite{NY20}). We emphasize here that each vertex group $G_v$ in a \CAT admissible groups maybe not split trivially as a product of a hyperbolic space with $\mathbb{Z}$. We need to pass to a further finite index subgroup to make our group is in a simpler form. By Lemma~7.2 in \cite{HNY21}, there exists a finite index subgroup of the \CAT admissible group $G$ such that this subgroup of $G$ preserves $\mathcal{V}_1$, $\mathcal{V}_2$, and each vertex group of the subgroup is a product of a nonelementary hyperbolic group with $\mathbb Z$.  By abuse of notations, we still denote this subgroup by $G$ (since HHS is preserved under quasi-isometry \cite{BHS19}).

\subsection{Constructing an index set}
\subsubsection{Elements in index set $\Lambda$: A collection of coned-off hyperbolic spaces}

Recall from Section~\ref{admissible} that for each vertex $v$ in the Bass-Serre tree $T$, there is associated block $B_{v} = H_v \times \mathbb{R}$. 
For a line $\ell$ and a constant $r>0$, a {\textit{hyperbolic $r$--cone}}, denoted by $cone_r(\ell)$   is the quotient space of $\ell\times [0, r]$ by collapsing $\ell\times 0$ as a point called \textit{apex}. A metric   is endowed  on $cone_r(\ell)$  so that it  is isometric to the metric completion of the universal covering of   a closed   disk of radius $r$ in the real hyperbolic plane $\mathbb H^2$ punctured at the center (see \cite[Part I, Ch. 5]{BH99}).

Given a vertex $v$ in the tree $T$. Let $
\mathbb{L}_{v} = \{ \ell_e = F_{e} \cap H_v  \, |\, e_{-} = v \}
$ be the set of boundary lines in the hyperbolic space $H_v$ (see Definition~\ref{defn:boundaryline}). We attach a hyperbolic $r$--cone to each boundary line $\ell_e$. The resulted space is denoted by $\hat{H}_v$. 

Since $\mathbb L_v$ has the bounded intersection property, it follows from Corollary~5.39 in \cite{DGO} that for a sufficiently large constant $r$, the space {$\widehat{H}_v$} is a hyperbolic space. Since there are finitely many vertices up to $G$--action, the constant $r$ can be chosen large enough so that $\widehat{H}_v$ is a $\delta(r)$--hyperbolic space for all vertex $v$ in $T$. 

The collection of $\delta(r)$--hyperbolic spaces $\bigl \{ \widehat{H}_v \, \bigl |  v \in \mathcal{V}(T) \bigr \}$ is a part in the index set $\Lambda$ (in Definition~\ref{defn:HHS}) that we will define later.

\subsubsection{Elements in the index set $\Lambda$: a collection of quasi-lines}
We need to add another collection of hyperbolic spaces into the collection above as the following.
Let $e$ be an oriented edge in $T$ such that $e_{+} = v$.
Let $\mathfrak f_e$ be a line in the boundary plane $F_e$  parallel to the $\mathbb R$-direction in $B_v = H_v \times \mathbb R$. Recall that two lines  $\mathfrak f_e\ne \mathfrak f_{\bar e}$ in the plane $F_e$ must be intersect.  The   action of  $G_e$ on the set of lines orthogonal to $\mathfrak f_e$ induces an action of $G_e$  on $\mathfrak f_e$. {Moreover, this action of $G_e$ on $\mathfrak f_e$ extends to the boundary plane $F_e$ so that it commutes with the orthogonal projection to $\mathfrak f_e$. Note that this new action of $G_e$ on $F_e$ preserves $\mathfrak f_e$ so does not coincide with the original action on $F_e$.}





\begin{figure}[htb]
\centering 
 \resizebox{0.5\textwidth}{!}{
\begingroup%
  \makeatletter%
  \providecommand\color[2][]{%
    \errmessage{(Inkscape) Color is used for the text in Inkscape, but the package 'color.sty' is not loaded}%
    \renewcommand\color[2][]{}%
  }%
  \providecommand\transparent[1]{%
    \errmessage{(Inkscape) Transparency is used (non-zero) for the text in Inkscape, but the package 'transparent.sty' is not loaded}%
    \renewcommand\transparent[1]{}%
  }%
  \providecommand\rotatebox[2]{#2}%
  \newcommand*\fsize{\dimexpr\f@size pt\relax}%
  \newcommand*\lineheight[1]{\fontsize{\fsize}{#1\fsize}\selectfont}%
  \ifx\svgwidth\undefined%
    \setlength{\unitlength}{260.74816734bp}%
    \ifx\svgscale\undefined%
      \relax%
    \else%
      \setlength{\unitlength}{\unitlength * \real{\svgscale}}%
    \fi%
  \else%
    \setlength{\unitlength}{\svgwidth}%
  \fi%
  \global\let\svgwidth\undefined%
  \global\let\svgscale\undefined%
  \makeatother%
  \begin{picture}(1,1.29837718)%
    \lineheight{1}%
    \setlength\tabcolsep{0pt}%
    \put(0,0){\includegraphics[width=\unitlength,page=1]{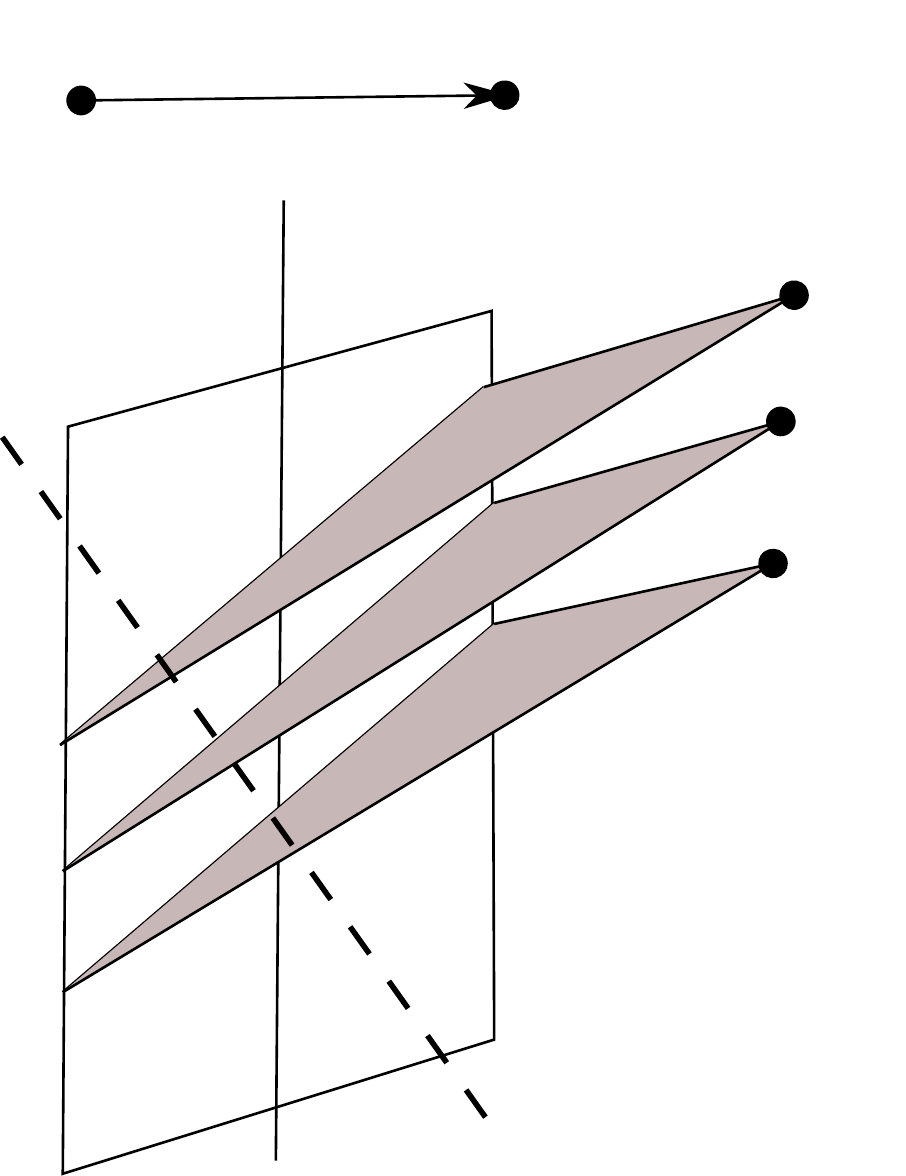}}%
    \put(0.04318109,1.26296226){\color[rgb]{0,0,0}\makebox(0,0)[lt]{\lineheight{1.25}\smash{\begin{tabular}[t]{l}\Large $v$\end{tabular}}}}%
    \put(0.54583959,1.25424566){\color[rgb]{0,0,0}\makebox(0,0)[lt]{\lineheight{1.25}\smash{\begin{tabular}[t]{l}\Large $w$\end{tabular}}}}%
    \put(0.28143541,1.22809579){\color[rgb]{0,0,0}\makebox(0,0)[lt]{\lineheight{1.25}\smash{\begin{tabular}[t]{l}\Large $e$\end{tabular}}}}%
    \put(0.55165075,0.04263515){\color[rgb]{0,0,0}\makebox(0,0)[lt]{\lineheight{1.25}\smash{\begin{tabular}[t]{l}\Large  $H_w \cap F_{e}$\end{tabular}}}}%
    \put(0,0){\includegraphics[width=\unitlength,page=2]{thick.pdf}}%
    \put(0.139064,0.9142974){\color[rgb]{0,0,0}\makebox(0,0)[lt]{\lineheight{1.25}\smash{\begin{tabular}[t]{l}$F_e$\end{tabular}}}}%
  \end{picture}%
\endgroup%
}
\caption{Dashed-line is the boundary line $\ell_w: = H_w \cap F_e$ of $H_w$. Each fiber line in the collection $\mathbb K_{\bar e}$ which is perpendicular to $H_w \cap F_e$ is coned-off.}
\label{figure:thick}
\end{figure}

We choose a  $Z(G_v)$--invariant set $\mathbb K_{\bar e}$ of fiber lines of $B_w$  in $F_e$ orthogonal to the boundary line $H_w \cap F_e$ of $H_w$ in $F_e$. A hyperbolic cone of radius $r=1$ is then attached to $F_e$ along each line in $\mathbb K_{\bar e}$. The resulted coned-off space denoted by $\widehat{F}_e$. 
We note that    $\widehat {F}_e$ is quasi-isometric to a line and thus it is a hyperbolic space.

Let $Lk(v)$ and $St(v)$ be the link and the star of $v$ in $T$. Then $Fl(v):=St(v)\times \mathbb R$ is the union of flat strips $S_e$'s (of width 1) corresponding to the set of oriented edges $e=[w,v]$  towards $v$. 
\begin{figure}[htb] 
\centering \scalebox{0.8}{
\includegraphics{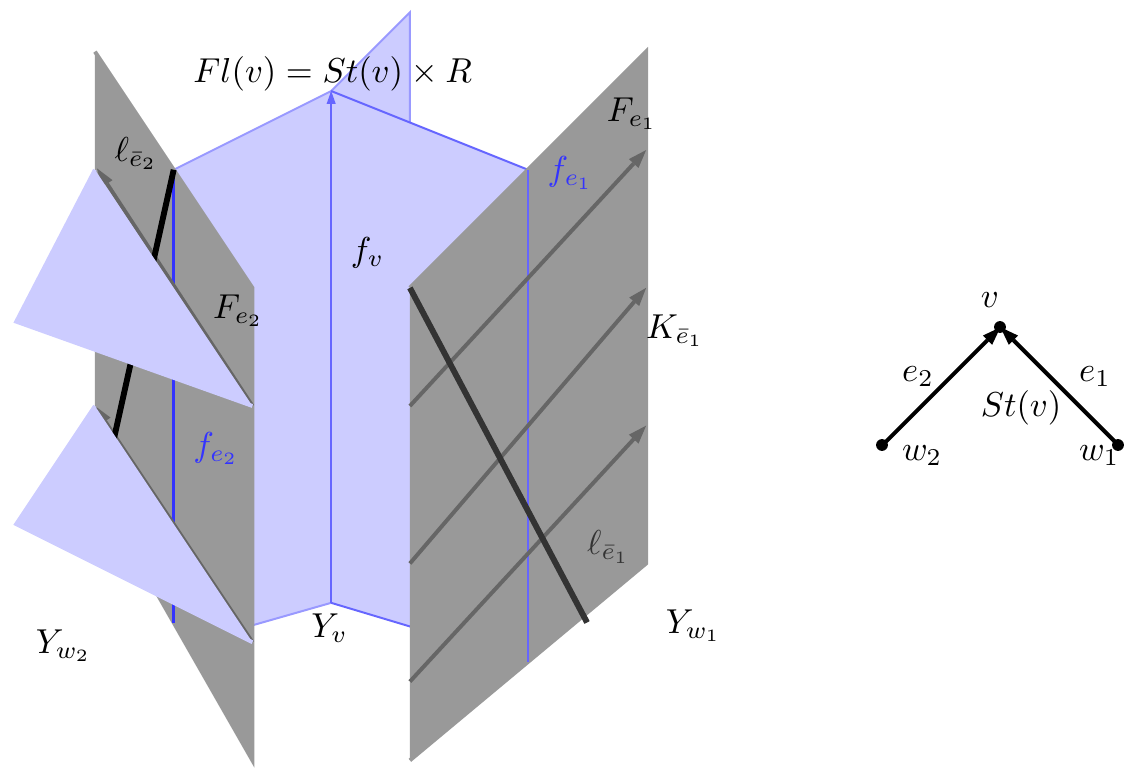} 
} \caption{The picture illustrates the space $\mathcal{R}_v$ is obtained by identifying boundary lines of strips $S_e$'s with marked lines $\frak{f}_e$ in $\widehat{F}_e$. The set of hyperbolic cones attached along lines in $\mathbb K_{\bar e_1}$ not shown in the figure. { The resulted space $\mathcal{R}_v$ is a quasi-line.}} \label{figure2}
\end{figure} 
Let $\mathcal{R}_v$ be the space obtained from the disjoint union   $Fl(v) \sqcup_{e\in St(v)} \widehat{F}_{e}$  with the boundary lines of every strip $S_{e}$ in $Fl(v)$ are identified with the marked  lines $\mathfrak f_e$ in $\widehat{F}_e$.  Endowed with length metric, $\mathcal{R}_v$ is  quasi-isometric  a quasi-line, so it is a hyperbolic space. We refer the reader to Figure~\ref{figure2} for this illustration.

\begin{remark}
\label{remark:useful}
By our construction of $\mathcal{R}_{v}$'s, there exists a uniform constant $r >0$ such that the following holds.
\begin{enumerate}
\item $\mathcal{R}_v$ is a $r$--hyperbolic space for every vertex $v$ in $T$.
    \item Let $v$ be an arbitrary vertex in $T$, and let $w$ be a vertex that is adjacent to $v$. Denote the edge joining two vertices $v$ and $w$ by $e$. Let $F_{e}$ be the boundary plane in $B_v$ associated to the edge $e$. Let $\ell$ be an arbitrary line in the plane $F_e$ that is parallel to the $\mathbb{R}$--factor of the adjace pice $B_w = H_{w} \times \mathbb{R}$ of $B_v$. Then the diameter of $\ell$ in $(\mathcal{R}_v, d_{\mathcal{R}_v})$ is no more than $r$.
    \item Let $\ell$ and $\ell'$ be two lines in the plane $F_e$ that is parallel to the $\mathbb{R}$--factor of the adjacent piece $B_w$. These two lines intersect the boundary line $\ell_{w} := F_{e} \cap H_w$ at two points $x$ and $y$ respectively. Then $\bigl | \diam_{(\mathcal{R}_v, d_{\mathcal{R}_{v})}} (\ell \cup \ell') - d(x, y) \bigr | \le 2r$.
\end{enumerate}
\end{remark}

\begin{definition}[The index set $\Lambda$]
We define the index set $\Lambda$ to be the union of the Bass-Serre tree $T$, the coned-off spaces $\widehat{H}_v$, and the thickened fiber lines $\mathcal{R}_v$, i.e,
$$
\Lambda : = \{T\} \cup \bigl \{\mathcal{R}_{v} : v \in V(T) \bigr \} \cup \bigl \{\widehat{H}_v : v \in V(T) \bigr \}
$$
\end{definition}





\subsection{Verifying the HHS Axioms}
To prove Proposition~\ref{prop:part1}, we need to associate to each element in this index  set $\Lambda$ a hyperbolic space, and verify that this index set together with these hyperbolic spaces satisfy HHS's axioms in Definition~\ref{defn:HHS}.

Recall that $G \curvearrowright X$ is a \CAT admissible action. Since HHS preserves under quasi-isometries (see Proposition~1.10 in \cite{BHS19}), it suffices to show that the orbit space $G(x_0) \subset X$ is a HHS for some $x_0 \in X$. Thus we assume, without loss
of generality that $x_0$ belongs to a plane $F_{e_0}$ for some edge $e_0$ in $T$. By abusing notation, {\it we still denote $G(x_0)$ by $X$.}

First at all, each element $W$ in the index set $\Lambda$ is associated to a hyperbolic space as the following: 
If $W = T$, we let $\calC W$ denote the Bass-Serre tree $T$ equipped with the metric $d_T$.
 If $W = \mathcal{R}_v$ for some vertex $v$, we let $\calC W$ denote the space $\mathcal{R}_v$ equipped  with the metric $d_{\mathcal{R}_v}$ described early. We recall that $ \bigl ( \mathcal{R}_v, d_{\mathcal{R}_v} \bigr )$ is quasi-isometric to a line, and thus $ \bigl ( \mathcal{R}_v, d_{\mathcal{R}_v} \bigr )$ is a hyperbolic space.
    If $W = \widehat{H}_v$ for some vertex $v$, we let $\calC W$ denote the space $\widehat{H}_v$ equipped with the coned-off metric $d_{\widehat{H}_v}$ described early in previous paragraphs.
    Note that these spaces are all hyperbolic spaces, and there are finitely many vertices up to $G$ action, thus there exists a constant $\delta >0$ such that $\calC W$ is a  $\delta$--hyperbolic space for all $W$ in the index set $\Lambda$.

 \subsubsection{ Projections}
\label{projectionaxiom} 
  In this section we are going to verify the Projection Axiom of Definition~\ref{defn:HHS}. Recall that $x_0$ belongs to a plane $F_{e_0}$ for some edge $e_0$ in $T$.  Let $W$ be an element in the index set $\Lambda$, we are going to define {\it projections}  $\pi_{W} \colon X \to 2^{CW}$ as follows.

 Assume that $W =T$, then we define the $\pi_{W} = \pi_T$ to be the index map $\pi \colon X \to T$ given by Remark~\ref{rem:indexfunction}.

     Now, assume that $W = \widehat{H}_v$ for some vertex $v \in T$.  then we define the projection $\pi_W \colon X \to 2^{CW}$ to be the nearest point projection from $X$ into $H_v \subset \widehat{H}_v$.

     Suppose that $W = \mathcal{R}_v$ for some vertex $v \in V(T)$. We are going to define the projection $\pi_{X} \colon X \to 2^{CW} = 2^{\mathcal{R}_v}$ as the following:
     Let $x$ be an arbitrary element in $X$, then $\pi(x)$ is a vertex in the tree $T$. We note that $x \in Y_{\pi(x)}$. If $\pi(x) = v$ then we define $\pi_{W}(x) : = x$. If $\pi(x)$ and $v$ are two adjacent vertices then we let $\bar{x}$ be the projection of $x$ into be base $H_{\pi(x)}$ of $Y_{\pi(x)} = H_{\pi(x)} \times \mathbb{R}$. Let $\alpha$ be a shortest geodesic in $H_{\pi(x)}$ joining $\bar{x}$ to the boundary line $\ell_{e} : = F_{e} \cap H_{\pi(x)}$ where $e : = [v, \pi(x)]$. Let $\mathcal{S}_{\bar{x}}$ be the strip in $Y_{\pi(x)}$ defined by $\alpha \times \mathbb{R}$. We then define $\pi_{W}(x)$ to be the subset $\mathcal{S}_{\bar{x}} \cap F_{e}$ of $F_{e} \subset \mathcal{R}_v$. If $d_{T}(v, \pi(x)) \ge 2$, then we let $\gamma$ be the geodesic in $T$ connecting $\pi(x)$ to $v$. We write $\gamma = e_{1} \cdots e_{k-1}e_{k}$ with $(e_1)_{-} = \pi(x)$ and $(e_k)_{+} = v$, then we define $\pi_{W}(x) : = \mathcal{S}_{e_{k-1}, e_{k}} \cap F_{e_k} \subset \mathcal{R}_{v}$. We remind here that the intersection $\mathcal{S}_{e_{k-1}, e_{k}} \cap F_{e_k}$ is a line in the plane $F_{e_k}$, but this line in the coned-off space $\mathcal{R}_v$ is no longer a line, and it is a bounded set in $\mathcal{R}_v$ (see Remark~\ref{remark:useful}).

To verify the HHS Axioms in Definition~\ref{defn:HHS}, we will use the following lemma.


\begin{lemma}
\label{lem:K}
There exists a uniform constant $A >0$ such that the following holds. Let $u$, $v$, and $w$ be vertices in $T$ such that $u, v \neq w$, and $u, v, w \in lk(o)$ for some vertex $o$. Let $e = [w,o]$, $e_1 = [u,o]$ and $e_2 = [v,o]$. Let $\ell_e$, $\ell_{e_1}$, $\ell_{e_2}$ be boundary lines in $H_v$ associated to the edges $e$, $e_1$, and $e_2$ respectively. Suppose that $x$ and $y$ are two points in $X$ such that $\pi(x) =u$ and $\pi(y) =v$. Then 
\[
\diam \bigl (\pi_{\mathcal{R}_w}(x,y)  \bigr ) \le A d_{\ell_{e}}(\ell_{e_1}, \ell_{e_2}) + A
\] where $\pi_{\mathcal{R}_w}(x,y) := \pi_{\mathcal{R}_{w}}(x) \cup \pi_{\mathcal{R}_{w}}(y)$ and $d_{\ell_{e}}(\ell_{e_1}, \ell_{e_2}) : = \diam \pi_{\ell_{e}}(\ell_{e_1}) \cup \pi_{\ell_{e}}(\ell_{e_1})$ with $\pi_{\ell_e} (\ell_{e_i})$ is the projection from $\ell_{e_i}$ into $\ell_e$ in the CAT(0) hyperbolic space $H_o$.
\end{lemma}
\begin{proof}
If $u =v$ then the result is vacuously true since both $\pi_{\mathcal{R}_{w}}(x)$ and $\pi_{\mathcal{R}_{w}}(y)$ are the same. We thus assume that $u \neq v$. By our definition of projections, we have $\pi_{\mathcal{R}_{w}}(x) = \mathcal{S}_{e_1 e} \cap F_e$ and $\pi_{\mathcal{R}_{w}}(y) = \mathcal{S}_{e_2 e} \cap F_e$. Note that $\pi_{\mathcal{R}_w}(x)$ and $\pi_{\mathcal{R}_w}(y)$ are parallel lines in $F_e$. These two lines intersect the boundary line $\ell_e$ at two points denoted by $\bar{x}$ and $\bar{y}$. By Remark~\ref{remark:useful}, we have that
\[
\bigl | \diam_{(\mathcal{R}_{w}, d_{\mathcal{R}_{w}})} \bigl (\pi_{\mathcal{R}_w}(x,y)  \bigr ) - d(\bar{x}, \bar{y}) \bigr | \le 2r
\]
We note that $\bar{x}$ and $\bar{y}$ are within a uniformed bounded distance from $\pi_{\ell_e} (\ell_{e_1})$ (projection from $\ell_{e_1}$ into $\ell_e$ in the CAT(0) hyperbolic space $H_o$) and $\pi_{\ell_e} (\ell_{e_2})$ respectively. It follows that there exists a uniform constant $A >0$ such that the conclusion of the lemma holds.


\begin{figure}[h!]
\begin{tikzpicture}
 \tikzstyle{vertex} =[circle,draw,fill=black,thick, inner sep=0pt,minimum size=.5 mm]
 
[thick, 
    scale=1,
    vertex/.style={circle,draw,fill=black,thick,
                   inner sep=0pt,minimum size= .5 mm},
                  
      trans/.style={thick,->, shorten >=6pt,shorten <=6pt,>=stealth},
   ]

    \draw[thick] (3,0) arc (0:360:3);
    \draw[thick] (2.45, 1.7) arc (121:240:2);
      \draw[thick] (-2.45, -1.7) arc (300:420:2);
      \draw[thick] (-1.5, 2.6) arc (210:322:2);
      
       \draw[thick, ->] (1.9, 1.2) arc (237:215:2.5);
       
       \draw[thick, ->] (-2.1, 1.5) arc (300:317:4);

          \node[vertex] (a) at (4, 0) [label=above right: $u$]  {};
          
           \node at (1.6, 0) [label=left : $\ell_{e_1}$]  {};
           \node at (-1.5, 0) [label=right : $\ell_{e_2}$]  {};
            \node at (0, 1.7) [label=below : $\ell_{e}$]  {};
            \node at (0, -0.7) [label=below : $H_0$]  {};

          \node at (5, 0) [label=above : $e_1$]  {};
          
          \node at (7, 0) [label=above : $e_2$]  {};
          \node at (6, 1) [label=right : $e$]  {};
          
          \node[vertex] (b) at (6, 0) [label=above right: $o$]  {};
          \node[vertex] (c) at (8, 0) [label=above right: $v$]  {};
         \node[vertex] (d) at (6, 2) [label=above right: $w$]  {};

  \draw [thick,->](a)--(b);
  \draw [thick,->] (d)--(b);
  \draw [thick,->](c)--(b);
  
\end{tikzpicture}
\caption{Two arrows are within a uniform Hausdorff distance from projections of $\mathcal{S}_{e_1e}$ and $\mathcal{S}_{e_2e}$ into $H_o$}
\label{Fig:YoYo}
\end{figure}
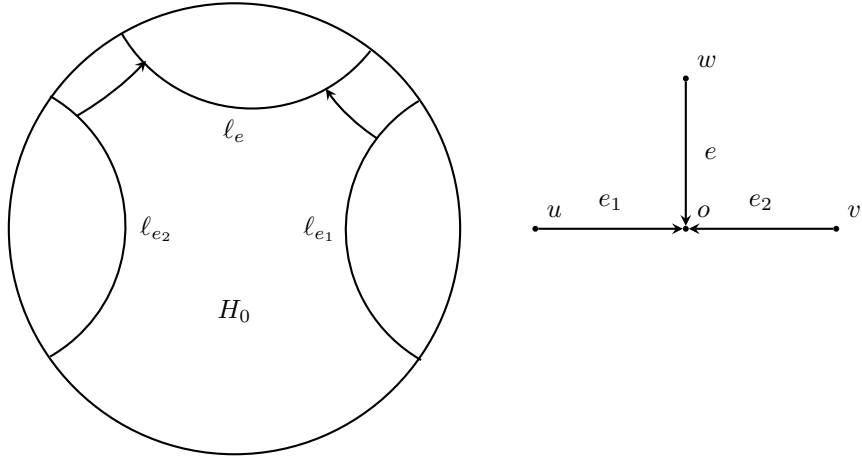

\end{proof}

\begin{lemma}
$\bigl (X, \Lambda, \{\pi_{W} \} \bigr ) $ satisfies the Projection Axiom.
\end{lemma}
\begin{proof}
By our definition of projection maps in Section~\ref{projectionaxiom}, it is obvious that $\diam_{\calC W} \bigl (\pi_{W}(x) \bigr ) \le \xi$ for some uniform constant $\xi >0$. Note that each $\pi_{W}$ is coarsely Lipschitz map. Indeed, this fact is clear when $W \in \Lambda \backslash \{T\}$. In the case $W =T$ then the coarse Lipschitz property of $\pi_W$ is followed from Remark~\ref{rem:indexfunction}. Also,  it is clear from our definition of projection maps that $\pi_{W}(X)$ is $K$--quasiconvex in $\calC W$ for some constant $K$ large enough. Hence Projection Axiom is verified. 
\end{proof}

 \subsubsection{Nesting}
 \label{nesting}
 We equip the index set $\Lambda$ with a partial order $\sqsubseteq$ as follows:
 We declare $W \sqsubseteq W$ (i.e, $W$ is nested in itself) for all $W$ in the index set $\Lambda$.
    For each $W \in \Lambda$, we declare that $W \sqsubseteq T$ (i.e, every element in $\Lambda$ is nested in $T$).
     Let  $V, W$ be distinct elements in $\Lambda \backslash \{T\}$. We say $V$ is {\it properly nested} in $W$ if $\calC W = \widehat{H}_v$ for some $v \in V(T)$ and $\calC V= \mathcal{R}_{w}$ for some vertex $w \in T$ such that $v$ and $w$ are adjacent vertices.

 \begin{lemma}
  $\bigl ( X, \Lambda, \sqsubseteq, \{ \pi_W \} \bigr )$ satisfies the Nesting Axiom.
 \end{lemma}
 
 \begin{proof}
By our definition of $\sqsubseteq$, every element in the index set is nested in $T$, hence the element $T \in \Lambda$ is the unique $\sqsubseteq$--maximal element. For each $W$ in the index set $\Lambda$, we denote by $\Lambda_{W}$ the set of $V \in \Lambda$ such that $V \sqsubseteq W$. 

Assume that $V, W \in \Lambda$ so that $V$ is properly nested in $W$. We are going to define a specific subset $\rho_{W}^{V} \subset \calC W$ such that $\diam (\rho_{W}^{V}) \le \xi$.   Also, there is a projection $\rho_{V}^{W}$ from $\calC W$ to subsets of $\calC V$. Note that in the Definition~\ref{defn:HHS}, there is no restriction on $\rho_{V}^{W}$.
If $W = \mathcal{R}_v$ then there is no $V$ such that $V$ is properly nested in $W$. Thus, we only need to consider the following cases:
 
{\textbf{Case~1}}: Assume that $W$ is the Bass-Serre tree $T$. In this case, we have that $\calC W = (T, d_T)$. Since $V$ is properly nested in $W$, it follows that $V \in \Lambda \backslash \{T\}$. Hence $\calC V$ is either the coned-off space $\bigl (\widehat{H}_v, d_{\widehat{H}_v} \bigr )$ or the thickened fiber line $\bigl (\mathcal{R}_v, d_{\mathcal{R}_v} \bigr ) $ for some vertex $v$ in $T$.  We thus define $\rho_{W}^{V}$ to be the vertex $v \in T = CW$. Obviously,  $\diam(\rho_{W}^{V}) = 0 < \xi$.

A projection map $\rho_{V}^{W} \colon \calC W = (T, d_T) \to 2^{\calC V}$ from $\calC W$ to subsets of $\calC V$ is defined as follows. Let $x$ be an arbitrary vertex in the Bass-Serre tree $T$. 
 If $d_{T}(x, v) \le 1$, then {\it we define $\rho^{W}_{V}(x)$ to be the entire space $\calC V$}.

    Now we assume that $d_{T}(x, v) \ge 2$. Let $e_1$ and $e_2$ be the last two consecutive oriented edges in the geodesic $[x,v]$ in $T$ such that $v = (e_{2})_{+}$. Let $u$ be the common vertex of $e_1$ and $e_2$, and let $\mathcal{S}_{e_1e_2}$ be the strip in the space $B_{u}$ (see Subsection~\ref{admissible}).
         Recall that $V$ is either $\widehat{H}_v$ or $\mathcal{R}_v$. When $V = \mathcal{R}_{v}$,  then we define $\rho^{W}_{V}(x)$ to be $\mathcal{S}_{e_1e_2} \cap F_{e_2} \subset \mathcal{R}_v$ . Let $r >0$ be the constant given by Remark~\ref{remark:useful}. Applying Remark~\ref{remark:useful} to $\ell: = \mathcal{S}_{e_1e_2} \cap F_{e_2}$, we have that the diameter of $\rho^{W}_{V}(x) = \ell$ in $\bigl ( \mathcal{R}_v, d_{\mathcal{R}_v} \bigr )$ is bounded above by $r$.   
When $ V = \widehat{H}_v$ then we { \it define $\rho^{W}_{V}(x)$ to be the apex-point} $c_{e_2}$ in $\widehat{H}_v$.

{\textbf{Case~2}}:   The last case we should consider is that $W = \widehat{H}_w$. Since $V$ is properly nested in $W$, it follows that $V = \mathcal{R}_{v}$ for some vertex $w$ where $v$ and $w$ are two adjacent vertices in $T$.
Let $e$ denote the edge $[w,v]$. We define $\rho^{V}_{W}$ to be the apex-point $c_{e}$ in $\widehat{H}_w$, and thus $\diam(\rho_{W}^{V}) = 0 < \xi$.  Now, we are going to define $\rho^{W}_{V} \colon \calC W \to 2^{\calC V} $ as follows. Let $x$ be a point in $\calC W  = \bigl ( \widehat{H}_w, d_{\widehat{H}_w} \bigr )$. If $x$ is the apex-point $c_e$ then we define $\rho^{W}_{V}(x)$ to be $\mathcal{R}_{v}$. Now we assume that $x$ is not the apex-point $c_e$. Let $\bar{x}$ the projection of $x$ to the subspace $H_v$ of $\widehat{H}_w$. Let $\mathcal{S}_{\bar{x}}$ be a strip in $B_{w}$ such that its projection into $H_w$ is a shortest path from $\bar{x}$ to the boundary line $\ell = F_{e} \cap H_w$. We define $\rho^{W}_{V}(x)$ to be the $\mathcal{S}_{\bar{x}} \cap F_{e} \subset \mathcal{R}_{v}$.
\end{proof}

 \subsubsection{Orthogonality}
 \label{orthogonality}
 We define orthogonality $\perp$ as the following:
 \begin{itemize}
     \item We declare $\mathcal{R}_{v} \perp \widehat{H}_v$ for all vertex $v$ in $T$.
     \item We declare $\mathcal{R}_{v} \perp \mathcal{R}_{w}$ whenever $v$ and $w$ are adjacent vertices.
 \end{itemize}
 
 \begin{lemma}
 $\bigl ( X, \Lambda, \perp \bigr )$ satisfies Orthogonality Axiom.
 \end{lemma}

\begin{proof}
Let $V$, $W$, and $U$ be elements in the index set $\Lambda$ such that $V \sqsubseteq W$ and $W \perp U$. Note that $W \neq T$ since every element in $\Lambda$ is nested in $T$. If $V =W$ then it is obvious that $V \perp U$. We thus assume that $V$ is properly nested in $W$. 
It follows that $W = \widehat{H}_v$ for some vertex $v$ (as no element in $\Lambda$ is nested in $\mathcal{R}_v$, so $W$ could not be $\mathcal{R}_v$). Hence $V = \mathcal{R}_w$ for some adjacent vertex $w$ of $v$. Since $\mathcal{R}_{v} \perp \mathcal{R}_w$ by our definition above, it follows that $V \perp U$. Also, it is clear from our definition that if $V \perp W$, then $V, W$ are not $\sqsubseteq$--comparable. 

For the rest of the proof, we are going to verify the statement in Orthogonality Axiom (see Definition~\ref{defn:HHS}):  ``for each $S \in \Lambda$, for each $U \in \Lambda_{S}$ such that the set $\{ V \in \Lambda_{S} \, | \, V \perp U \} \neq \emptyset$ then there exists $W \in  \Lambda_{S}$, $W \neq S$ so that whenever $V \perp U$, $V \sqsubseteq S$ then we have $V \sqsubseteq W$''.

We observe that the element $S$ must be $T$. Indeed, suppose that $S \neq T$. In this case then $S$ is either $\widehat{H}_v$ or $\mathcal{R}_v$ for some vertex $v \in T$.
When  $S = \mathcal{R}_v$ then $\Lambda_S = \{ \mathcal{R}_v \}$. As $U$ and $V$ are in $\Lambda_S$, it follows that the set $\{ V \in \Lambda_{S} \, \bigl | \, V \perp U \} = \emptyset$.
Now we assume that $S = \widehat{H}_v$. As $U \in \Lambda_{S}$, we have that $U$ is either $\widehat{H}_v$ or  $U = \mathcal{R}_{w}$ for some vertex $w$ that is adjacent to $v$. If $U$ is $\widehat{H}_v$ then the set $\{ V \in \Lambda_{S}: V \perp U \} = \emptyset$. Now we consider $U = \mathcal{R}_{w}$. Since $V \in \Lambda_{S}$, it follows that $V = \widehat{H}_v$ or $V = \mathcal{R}_{a}$ where $a$ and $v$ are adjacent. Note that $\widehat{H}_v$ is not orthogonal to $U = \mathcal{R}_w$, and $V = \mathcal{R}_{a}$ is not orthogonal to $U = \mathcal{R}_w$ (since $a$ and $w$ are not adjacent vertices). In other words, the set $\{ V \in \Lambda_{S}: V \perp U \} = \emptyset$.


Therefore, the element $S = T$, and hence $\Lambda_S = \Lambda$. Since $U \in \Lambda_{S}$, it follows that $U$ is either $T$, or $\mathcal{R}_v$ or $\widehat{H}_v$ for some vertex $v \in T$. Since no element in $\Lambda$ is orthogonal to $T$, it follows that $U$ could not be $T$, otherwise the set $\{ V \in \Lambda_{S}: V \perp U \} = \emptyset$.
When 
$U =\mathcal{R}_{v}$ then we define $W$ to be  $\widehat{H}_v$. Note that whenever $V \perp U = \mathcal{R}_v$ and $V \sqsubseteq S = T$ then by  our definition of $\perp$, the element $V$ is either $\widehat{H}_v = W \sqsubseteq W$ or $V = \mathcal{R}_w$ for some vertex $w$ that is adjacent to $v$. By definition of $\sqsubseteq$ in Section~\ref{nesting}, the element $V$ is nested in $W$.
When $U = \widehat{H}_v$ then we choose $W$ to be the $\mathcal{R}_v$ in the index set $\Lambda$. It is obvious that if $V \perp U = \widehat{H}_v$ and $V \sqsubseteq S =T$ then $V$ must be $\mathcal{R}_v = W \sqsubseteq W$. 
\end{proof}

\subsubsection{Transversality and consistency}
\label{Transversality and consistency}
If $V$ and $W$ are not orthogonal and neither is nested in the other, then we say that $V$ and $W$ are {\it transverse}, denoted $
V \pitchfork W
$. We have the following possibilities for $V$ and $W$.

{\it Case~1:} $V = \mathcal{R}_v$ and $W = \mathcal{R}_w$ for some vertices $v$ and $w$ in $T$. We note that the distance between $v$ and $w$ in $T$ is at least $2$. Otherwise, $V \perp W$ if $v$ and $w$ are adjacent and $V =W$ if $v =w$.
    
{\it Case~2:}  $V = \mathcal{R}_{v}$ and $W = \widehat{H}_w$ for some vertices $v, w \in T$. Again, we note that the distance between $v$ and $w$ is at least $2$. Indeed, if $v$ and $w$ are adjacent vertices then $\mathcal{R}_{v}$ is nested in $\widehat{H}_w$ by our definition which contradicts to the fact $V$ and $W$ are transverse. If $v =w$ then $V =W$ that contradicts to the fact $V$ and $W$ are transverse.
    
{\it Case~3:} $V = \widehat{H}_v$ and $W = \widehat{H}_w$ for some distinct vertices $v, w \in T$.


Regarding Case~1:
 Let $e_{1} e_{2} \dots e_{k-1}e_{k}$ ($k \ge 2$) be the geodesic edge path in the Bass-Serre tree $T$ connecting $v$ to $w$ with $v = (e_{1})_{-}$ and $(e_{k})_{+} = w$. Let $F_{e_k}$ and $F_{e_1}$ be the boundary planes (see Subsection~\ref{admissible}) of the blocks $B_w$ and $B_v$ respectively. Let $\mathcal{S}_{e_{k-1}e_{k}}$ and $\mathcal{S}_{e_1e_2}$ be the given strips (see Subsection~\ref{admissible}) in the blocks $B_{(e_k)_{-}}$ and $B_{(e_1)_{+}}$ respectively.

 We define
$$P^V_W : = \mathcal{S}_{e_{k-1}e_{k}} \cap F_{e_k}$$ and 

$$P^{W}_{V} : = \mathcal{S}_{e_{1}e_{2}} \cap F_{e_1}$$

 \begin{lemma}
 $V, W, P^{V}_{W}, P^{W}_{V}$ defined in Case~1 above satisfying  \eqref{eqn:1}, in Definition~\ref{defn:HHS}.
 \end{lemma}
 
 \begin{proof}
 Let $r >0$ be the constant given by Remark~\ref{remark:useful}. We have that $\diam P^{V}_{W} \le r \le \xi$ and $\diam P^{W}_{V} \le r \le \xi$.
  As we have shown early, the distance between $v$ and $w$ is at least $2$.
Let $u = \pi(x) \in V(T)$ where $\pi \colon X \to T$ be the index map given by Remark~\ref{rem:indexfunction}.

Suppose that $u$ does not belong to the geodesic $[v, w]$. If $v$ lies between $u$ and $w$ then by the definition of $\pi_W$ (see Subsection~\ref{projectionaxiom}) and $P^V_W$, we see that $\pi_{W}(x)$ and $P^{V}_{W}$ are the same subset in $\calC W$.  Similarly, if $w$ lies in $[v,u]$ then $\pi_{V}(x)$ and $P^{W}_{V}$ are the same subset in $\calC V$. By Remark~\ref{remark:useful}, we have that \eqref{eqn:1} holds.

Now, we assume that $u$ belongs to $[v,w]$. Suppose that $d(u,v) \ge 2$ (resp.\, $d(u, w) \ge 2$) then we have that $\pi_{V}(x)$ and $P^{W}_{V}$ are the same subset in $\calC V$ (resp,\, $\pi_{W}(x)$ and $P^{V}_{W}$ are the same subset in $\calC W$). It follows from Remark~\ref{remark:useful} that \eqref{eqn:1} holds whenever $d(u,v) \ge 2$ or $d(u, w) \ge 2$.

The final case we should consider is that $d(v,w) =2$ and $u$ lies between $[v,w]$. Let $\mathcal{S}: = \mathcal{S}_{[v,u], [u,w]}$ be the strip in $B_u$. Since $x \in X = G(x_0)$, and $x_0$ is in the boundary plane $F_{e_0}$, it follows that $x$ must lie in a boundary plane of $B_u$. We denote this boundary plane by $F_{e'}$. If $e' = [v,u]$ (the case $e' = [u,w]$ is similar) then $\pi_{W}(x)$ and $\mathcal{S} \cap F_{[u,w]}$ are identical and their diameter is no more than $r$ in $\calC W$ by Remark~\ref{remark:useful}. Now we assume that the edges $e$, $[v,u]$, and $[u.w]$ are distinct.

We endow each $H_v$ with the HHS structure originating from the fact that $H_v$ is hyperbolic relative to the collection of boundary lines $\mathbb{L}_v$ (see Lemma~\ref{numberlines}). As a result, The Transversality Axiom and the Consistency Axiom is applied to this HHS structure. In particular,
\[
\min \bigr \{ \pi_{\ell_1}(\ell, \ell_2), \pi_{\ell_2}(\ell_1, \ell) \bigr \} \le \lambda
\] for any boundary line $\ell \in \mathbb{L}_u$.
It implies that \eqref{eqn:1} holds.
 \end{proof}

 Regarding Case 2:
Assume that $V = \mathcal{R}_{v}$ and $W = \widehat{H}_w$ for some vertices $v, w \in T$. 
Since $V$ and $W$ are transverse, it follows that $v \neq w$ (otherwise they are orthogonal).
Note that $d(v, w) \ge 2$. Indeed, if $v$ and $w$ are adjacent vertices then $\mathcal{R}_{v}$ is nested in $\widehat{H}_w$ by our definition which contradicts to the fact $V$ and $W$ are transverse. Now, let $e_{1} e_{2} \dots e_{k-1}e_{k}$ be the geodesic edge path in the Bass-Serre tree connecting $v$ to $w$ with $v = (e_{1})_{-}$ and $(e_{k})_{+} = w$. We define

$$P^{W}_{V} : = \mathcal{S}_{e_{1}e_{2}} \cap F_{e_1} \subset \mathcal{R}_{v}$$ and
$$P^V_W : =  c_{e_{k}}$$ the  apex-point in $\widehat{H}_w$. Hence, subsets $P^{W}_{V}$ and $P^{V}_{W}$ are bounded sets in $\mathcal{R}_v$ (see Remark~\ref{remark:useful}) and $\widehat{H}_w$ respectively. Therefore  $\diam(\rho^V_W) \le \xi$, and $\diam(P^{W}_{V}) \le \xi$.
 \begin{lemma}
 $V, W, P^{V}_{W}, P^{W}_{V}$ defined in Case~2 above satisfying  \eqref{eqn:1},  in Definition~\ref{defn:HHS}.
 \end{lemma}
 \begin{proof}
Let $u = \pi(x) \in V(T)$ where $\pi \colon X \to T$ be the index map given by Remark~\ref{rem:indexfunction}. If $v$ lies between $u$ and $w$ then by the definition of $\pi_{W}$ in Subsection~\ref{projectionaxiom}, we have that $\pi_{W}(x)$ is a point in the boundary line $H_{w} 
\cap F_{e}$ where $e$ is the last edge in the geodesic $[v,w]$. Since $P^{V}_{W}$ is the apex-point $c_e$, it follows that the distance between $\pi_{W}(x)$ and $P^{V}_{W}$ in $\widehat{H}_w$ is bounded above by $r$. Thus \eqref{eqn:1} holds. If  $w$ lies in $[v, u]$, then similar arguments as above show that $\pi_{W}(x)$ and $P^{W}_{V}$ are the same subset in $\calC V$, and thus \eqref{eqn:1} holds.  

Now we assume that $u$ belongs to $[v, w]$. Let $e$ be the last edge in the geodesic path $[v,w]$. Then $P^{V}_{W}$ is the apex-point $c_e$ in $\widehat{H}_w$ and $\pi_{W}(x)$ is a point in the boundary line $H_{w} \cap F_{e}$. Since the distance between $\pi_{W}(x)$ and $c_e$ in $\widehat{H}_w$ is bounded above by $r$. It follows that \eqref{eqn:1} holds.
 \end{proof}

Regarding Case~3: Finally, suppose that $V = \widehat{H}_v$ and $W = \widehat{H}_w$ for some vertices $v, w \in T$.
Since $V$ and $W$ are transverse so $v \neq w$. 
Let $e_{1}\dots e_{k-1}e_{k}$ be the geodesic edge path in the Bass-Serre tree connecting $v$ to $w$ with $v = (e_{1})_{-}$ and $(e_{k})_{+} = w$. We then define $ P^V_W : =  c_{e_{k}}$ and $P^{W}_{V} : =  c_{e_{1}}$ to be apex-points in $\widehat{H}_w$ and $\widehat{H}_v$ respectively. Therefore  $\diam(\rho^V_W) \le \xi$, $\diam(P^{W}_{V}) \le \xi$.

  \begin{lemma}
 $V, W, P^{V}_{W}, P^{W}_{V}$ defined in Case~3 above satisfying  \eqref{eqn:1} in Definition~\ref{defn:HHS}.
 \end{lemma}
 
 \begin{proof}
 $V = \widehat{H}_v$ and $W = \widehat{H}_w$ where $d(v, w) \ge 2$. The proof is similar as the previous paragraphs, and we leave it to the reader.
\end{proof}

\begin{lemma}
$V, W, P^{V}_{W}, P^{W}_{V}$  satisfy \eqref{eqn:2} in Definition~\ref{defn:HHS}
\end{lemma}


\begin{proof}
Let $u = \pi(x)$ where $\pi \colon X \to T$ be the index map given by Remark~\ref{rem:indexfunction}.

Case~1: Suppose that $W = T$.  Since $V$ is nested in $W$, it follows that $V = \mathcal{R}_{v}$ or $V = \widehat{H}_v$ for some vertex $v$ in $T$.
We note that if $d(u,v) \le 2$ then $$d_{W}(\pi_{W}(x), \rho^{V}_{W}) = d_{T}(\pi_{T}(x), \rho^{V}_{T}) =  d(u, v) \le 2$$ and thus \eqref{eqn:2} holds. 

Now we assume that $d(u,v) >2$.    Let $e$ and $e'$ be two consecutive edges in the geodesic $[u,v]$ such that $(e')_{+} = v$, $(e)_{+} = (e')_{-}$. 
 If $V = \mathcal{R}_v$ then by our definition of $\pi_V$ (see Subsection~\ref{projectionaxiom}), we have $\pi_{V}(x) = \mathcal{S}_{e,e'} \cap F_{e'}$ is a bounded subset in $\calC V$ by Remark~\ref{remark:useful},  and $$\rho_{V}^{W}(\pi_{W}(x)) = \rho_{V}^{T} (\pi(x)) =  \rho_{V}^{T}(u) = \mathcal{S}_{e,e'} \cap F_{e'}$$ which is a bounded subset in $\calC V$ by Remark~\ref{remark:useful}; thus \eqref{eqn:2} holds. 
If $V = \widehat{H}_v$ then $\pi_{V}(x)$ is a point in the boundary line $H_{v} \cap F_{e'}$ and $\rho^{W}_{V}(\pi_{W}(x))$ is the apex-point $c_{e'}$ in $ \widehat{H}_v$. Since the distance between $\pi_{V}(x)$ and $c_{e'}$ is bounded above by $r$ in $\widehat{H}_v$, it follows that \eqref{eqn:2} holds.

Case~2:  $W = \widehat{H}_w$ for some vertex $w$ in $T$. 
Since $V$ is nested in $W$, it follows that $V = \mathcal{R}_{v}$ for some vertex $v$ that is adjacent to $w$. We note that $P^{V}_{W}$ is the apex-point $c_{e}$ in $\widehat{H}_w$ where $e = [v,w]$.

When $v \in [u, w]$ then $\pi_{W}(x)$ is a point in the boundary line $H_{w} \cap F_{e}$ of $H_w$. In $\widehat{H}_w$, this point and the apex-point $c_e$ have a distance that is bounded above by $r$, thus \eqref{eqn:2} holds.

When $w \in [v, u)$, let $e'$ be the edge in the geodesic $[u, w]$ with $(e')_{+} = w$. We have 
$\pi_{V}(x) = \pi_{\mathcal{R}_{v}}(x) = \mathcal{S}_{e',e} \cap F_{e} \subset \mathcal{R}_v$.
By the definition of $\pi_{W}$, we have that $\pi_{W}(x)$ is a point in a boundary line $\ell$ of $H_w$. Note that this line is a subset of $F_{e'}$.  By Lemma~\ref{lemma:epsilon}, $\rho^{W}_{V}(\pi_{W}(x))$ and $\mathcal{S}_{e', e} \cap F_{e}$ are within an uniform finite Hausdorff distance in the plane $F_{e}$, and thus their distance is uniformly bounded with respect to the metric in $\mathcal{R}_v$. Thus \eqref{eqn:2} holds.
For the case $u =w$, then since $x$ lies in the orbit $G(x_0)$, so $x \in F_{f}$ for some edge $f$ with $(f)_{+} =v$. The proof in this case is identical as the previous case.
\end{proof}

So far, we have the notation $\rho^{W}_{V}$ (a map from $\calC W$ to subsets of $\calC V$) whenever $V$ is properly nested in $W$, and the notation $P^{W}_{V}$ (a bounded subset of $\calC V$) whenever $V  \pitchfork W$. 
For simplicity, we will write $\mathcal{P}^{W}_{V}$ to mean either $\rho^{W}_{V}$ or $P^{W}_{V}$ when it is clear from the context.

\begin{lemma}
$V, W, \mathcal{P}^{W}_{V}, \mathcal{P}^{V}_{W}$  satisfy \eqref{eqn:3} in Definition~\ref{defn:HHS}.
\end{lemma}

\begin{proof}

 We will assume that $U$ is properly nested in $V$ since if $U =V$ then there is nothing to show.

Assume that $V = T$, then $U$ is either $\mathcal{R}_v$ or $\widehat{H}_{v}$ for some vertex $v \in T$. The set of $W \in \Lambda$ satisfies either $V$ is properly nested in $W$ or $
V \pitchfork W
$ and $W$ is not orthogonal to $U$ is empty. The claim is vacuously true. Assume that $V = \mathcal{R}_v$ for some vertex $v$, then the set of $U \in \Lambda$ that is properly nested in $V$ is empty.  The claim is vacuously true. Assume that $V =  \widehat{H}_{v}$ for some vertex $v$, then since $U$ is properly nested in $V$, it follows that $U = \mathcal{R}_{w}$ for some vertex $w$ that is adjacent to $v$. 

When $V$ is properly nested in $W$ then $W = T$. In this case, we note that $\rho^{U}_{W} = \rho^{\mathcal{R}_{w}}_{T} = w$, and $\rho^{V}_{W} = v$, thus $d_{W}(\rho^{U}_{W}, \rho^{V}_{W}) = d_{T}(w, v) = 1 \le \kappa_0$. Hence, \eqref{eqn:3} holds. When $ \widehat{H}_v = 
V \pitchfork W
$ and $W$ is not orthogonal to $U = \mathcal{R}_w$. Then there are two cases we should consider:

Case~1: $W = \mathcal{R}_{u}$ with $u \neq v, w$, $d(u, w) \ge 2$. Note that $d(u,v) \ge 2$ (otherwise $d(u,v) =1$ then $\mathcal{R}_{u} \sqsubseteq \widehat{H}_v$, so $W \sqsubseteq V$, a contradiction). In this case, either $v \in [u,w]$ or $w \in [v,u]$. Let $\gamma$ be the geodesic in $T$ containing $[v,w]$ with the endpoint $u$. Let $e, e'$ be the two edges in $\gamma$ such that $(e)_{+} = (e')_{-}$ and $(e')_{+} =u$. We have $d_{W}(P^{U}_{W}, P^V_W) \le \kappa_{0}$ since both $P^{U}_{W}$ and $P^{V}_{W}$ are $\mathcal{S}_{e, e'} \cap \mathfrak{f}_{e'}$ that is a bounded subset in $\calC W$. Hence \eqref{eqn:3} holds.

Case~2: $W = \widehat{H}_u$ where $u \neq v, w$ and $d(u, w) \ge 2$. Let $\gamma$ be the geodesic in $T$ containing $[v,w]$ with the endpoint $u$. Let $e$ be the edge in $\gamma$ with $(e)_{+} = u$. Then $d_{W}(P^{U}_{W}, \rho^V_W) \le \kappa_{0}$ since $P^{U}_{W}$ and $P^{V}_{W}$ both are the apex-point $c_{e}$ and the lemma is proven.
\end{proof}

\subsubsection{ Finite complexity}
It follows from the construction that the complexity of $X$ is $3$. Indeed,  we have that $V \sqsubseteq W \sqsubseteq T$ whenever $V = \mathcal{R}_v$, $W = \widehat{H}_w$ for some adjacent vertices $v$ and $w$ in $T$.

\subsubsection{Large Links}
\begin{lemma}
\label{lem:largelink}
$(X, \Lambda)$ satisfies the Large Links Axiom.
\end{lemma}
\begin{proof}

  We endow each $H_v$ with the HHS structure originating from the fact that $H_v$ is hyperbolic relative to the collection of boundary lines $\mathbb{L}_v$ (see Lemma~\ref{numberlines}). As a result, Large Link Axiom is applied to this HHS structure. Since there are finitely many vertices up to $G$--action, there are uniform constants $\lambda \ge 1$, $L \ge 1$ such that for each vertex $v \in T$ the following holds: Let $x$ and $y$ be two points in $H_v$, and let $N' = \lambda d_{\widehat{H}_v} (\pi_{\ell}(x), \pi_{\ell}(x')) + \lambda$. Then there exists $\ell_1, \ell_2, \cdots, \ell_{\lfloor N' \rfloor}$ such that for all boundary lines $\ell \in \mathbb{L}_v$ then either $\ell = \ell_i$ for some $i$, or $d_{\ell} (\pi_{\ell}(x), \pi_{\ell}(x')) < L$. Also, $d_{\widehat{H}_v}(\pi_{\widehat{H}_v}(x), P^{\ell_{i}}_{\widehat{H}_v}) \le N'$ for each $i$.

Let $r>0$ be the constant given by Remark~\ref{remark:useful}. We could enlarge $\lambda$ if necessary and assume that $\lambda > 2r$.  We are going to verify the following statement in the Large Links Axiom of Definition~\ref{defn:HHS}.
``  Let $W \in \Lambda$ and let $x, x' \in X$. Let $N = \lambda d_{w} (\pi_{W}(x), \pi_{W}(x')) + \lambda$. Then there exists $A_1, A_2, \cdots, A_{\lfloor N \rfloor} \in \Lambda_{W} \backslash \{W\}$ such that for all $A \in \Lambda_{W} \backslash \{W\}$ then either $A = A_i$ for some $i$, or $d_{A} (\pi_{A}(x), \pi_{A}(x')) < E$. Also, $d_{W}(\pi_{W}(x), P^{A_{i}}_{W}) \le N$ for each $i$ ''.
 
 Recall that $x_0$ is in the boundary plane $F_{e_0}$ for some edge $e_0$. Since $X$ is the orbit $G(x_0)$ and the collection of boundary planes is $G$--equivariant, it follows that $x$ and $x'$ belong to boundary planes $F_{e}$ and $F_{e'}$ for some edges $e$ and $e'$. 
 
 We first observe that if $W = \mathcal{R}_{v}$ for some vertex $v \in V(T)$ then there is nothing to show since the set $\Lambda_{W} \backslash \{W\} = \emptyset$. Hence, in the rest of the proof, we only need to consider the cases $W = \widehat{H}_v$ and $W =T$.

{\it  Case~1:} Suppose that $W = \widehat{H}_v$ for some vertex $v \in T$. Let $A$ be an arbitrary element in  $\Lambda_{W} \backslash \{W\}$. As $A$ is nested in $W = \widehat{H}_v$ and $A \neq W$, it follows that $A = \mathcal{R}_w$ for some vertex $w$ which is adjacent to $v$.

 {\it Case~1.1:} $v$ is a vertex of both edges $e$ and $e'$. Without loss of generality, we assume that $(e)_{+} v$ and $(e')_{+} = v$. In this case, we note that $F_{e}$ and $F_{e'}$ are two boundary planes in the vertex space $B_{v} = H_{v} \times \mathbb{R}$.  Let $\ell_e$ and $\ell_{e'}$ be the boundary lines in $H_v$ corresponding to edges $e$ and $e'$ respectively. Let $\bar{x}$ and $\bar{x'}$ be the projection point of $x$ and $x'$ into the boundary line $\ell_e$ and $\ell_{e'}$ respectively.
 
 
 
 Each boundary line $\ell_j$ with $j \in \{1, \dots, k \}$ is associated to an edge $e_j$ with $(e_{j})_{+} = v$. 
  We denote $v_j$ by $(e_j)_{-}$. Let $A_{j} : = \mathcal{R}_{v_j}$ with $j \in \{1, \dots, k \}$. If the vertex $w$ is equal to $v_j$ for some $j \in \{1, \dots, k \}$ then $A = A_j$. Now we assume that $w \notin \{v_1, \dots, v_k\}$. Let $\ell$ be a boundary line of $H_v$ associated to the edge $[v,w]$. We have that $d_{\ell}(x, x') \le L$, otherwise $w = v_j$ for some $j$. Let $A >0$ be the constant given by Lemma~\ref{lem:K}. We have 
 \begin{align*}
     d_{A} (\pi_{A}(x), \pi_{A}(x')) &= 
diam \bigl (\pi_{\mathcal{R}_w}(\mathcal{R}_{(e)_{-}}, \mathcal{R}_{(e')_{-}})  \bigr ) \\
& \le A d_{\ell}(\ell_{e}, \ell_{e'}) + A \\
&\le A d_{\ell}(x, x') + A
 \le AL + A \le E
 \end{align*}

It follows from the second paragraph above that with each $i = 1, \dots, k$, we have
 \begin{align*}
     d_{W}(\pi_{W}(x), \rho^{\ell_{i}}_{W}) & = d_{\widehat{H}_{v}} \bigl (\pi_{\widehat{H}_{v}}(x), \rho^{A_{i}}_{\widehat{H}_v} \bigr )  \le 
     N' \le N
 \end{align*}

{\it Case~1.2:} Assume that  $(e)_{+}  = v$ and $v$ is not a vertex of $e'$ (the case $(e')_{+}  = v$ and $v$ is not a vertex of $e$ is proved similarly).

When $v$ lies between $(e)_{-}$ and $e'$ then the proof is identical as Case~1.1.

When $(e)_{-}$ lies between $v$ and $e'$ then we define $A_1 = A_2 = \dots, = A_{\lfloor N \rfloor} =  \mathcal{R}_{(e)_{-}}$. We note that $d_{W} \bigl (\pi_{W}(x), c_e \bigr ) = d_{\widehat{H}_v} \bigl (\pi_{\widehat{H}_V}(x), c_e \bigr ) \le r$, and $d_{\widehat{H}_v} \bigl ( P^{\mathcal{R}_{(e)_{-}}}_{\widehat{H}_v}, c_e \bigr ) \le r$, thus
\[
d_{W}(\pi_{W}(x), P^{A_{i}}_{W}) \le 2r \le  N
\]
Let $A$ be an arbitrary element in  $\Lambda_{W} \backslash \{W\}$. As $A$ is nested in $W = \widehat{H}_v$ and $A \neq W$, it follows that $A = \mathcal{R}_w$ for some vertex $w$ which is adjacent to $v$. If $A \notin \{A_1, \dots, A_{\lfloor N \rfloor} \}$ then $w \neq (e)_{-}$. Since $\pi_{A}(x') = \pi_{\mathcal{R}_w} (x)$ is $\mathcal{S}_{[v,w], e} \cap F_{[v,w]} \subset \mathcal{R}_w$, and $\pi_{A}(x)$ and $\mathcal{S}_{[v,w], e} \cap F_{[v,w]}$ are close within a uniform Hausdorff distance in the plane $F_{[v,w]}$, it follows that the distance of two bounded subsets (with respect to the metric of $\mathcal{R}_w$) $\pi_{A}(x)$ and $\pi_{A}(x')$ no more than $E$.

{\it Case~1.3:} Suppose that
 $F_{e}$ and $F_{e'}$ are not boundary planes of $B_{v}$. 
 
 Denote the last edge in the geodesic from $e$ to $v$ is $f$, and the last edge in the geodesic from $e'$ to $v$ is $f'$. Note that both $(f)_{+} = (f')_{+} =v$. We will consider the collection $\{A_i\}_{i=1}^{\lfloor N \rfloor} =  \{ \mathcal{R}_{(f)_{-}}, \mathcal{R}_{(f')_{-}} \}$. Similarly as in Case~1.2, if $A$ is nested in $W$ and $A \neq W$ then $A = \mathcal{R}_w$ for some vertex $w$ that is adjacent to $v$. If $A \notin \{A_1, \dots, A_{\lfloor N \rfloor} \}$ then $w \notin \{(f)_{-}, (f')_{-} \}$. It follows from Lemma~\ref{lem:K} that $d_{A} \bigl (\pi_{A}(x), \pi_{A}(x') \bigr ) < E$. 
 
For each $i$, since $\pi_{W}(x)$ is a point in the boundary line $\ell_{f}$ of $H_{v}$, and $P^{A_i}_{W}$ is the apex-point $c_{f}$, and thus $d_{W}(\pi_{W}(x), P^{A_i}_{W}) \le 2r \le N$.

{\it Case~2:} $W$ is the Bass-Serre tree $T$. Recall that $x \in F_e$ and $x' \in F_{e'}$.

Let $\alpha$ be the geodesic in $T$ joining the edge $e$ to $e'$. Let
\[
\{A_1 , A_2, \cdots, A_{\lfloor N \rfloor} \} = \{\mathcal{R}_v, \widehat{H}_v \, \bigl | \, v \, \text{is a vertex in}\, \alpha \}
\]
Since $A$ is properly nested in $W =T$, it follows that $A = \mathcal{R}_w$ or $\widehat{H}_w$ for some vertex $w$ in $T$. 

If $A \notin  \{A_1 , A_2, \cdots, A_{\lfloor N \rfloor} \} $ then $w \notin \alpha$. Using the same argumenrts as in the previous subcase, we have that $d_{A} (\pi_{A}(x), \pi_{A}(x')) < E$. For each $i \in \{1, \cdots, \lfloor N \rfloor \}$ then $A_i$ is either $\mathcal{R}_v$ or $\widehat{H}_v$ for some vertex $v \in \alpha$. Let $\pi \colon X \to T$ be the index map given by Remark~\ref{rem:indexfunction}. Recall that we define the projection $\pi_{T}$ given by $\pi_{T}(x) = \pi(x)$ for any vertex $x \in T$. Denote $\alpha(0)$ and $\alpha(1)$ be the initial and terminal vertices of the geodesic $\alpha$. Since $\alpha$ is a geodesic joining $e$ to $e'$ and $x \in F_{e}$, $x' \in F_{e'}$, it follows that
$d(\pi(x), \alpha(0)) \le 2$ and $d(\pi(x'), \alpha(1)) \le 2$.
As $\rho^{A_i}_{W} = \rho^{A_i}_{T} =v$, we have
\begin{align*}
    d_{W} \bigl (\pi_{W}(X), \rho^{A_i}_{W} \bigr ) &= d_{T}\bigl ( \pi(x), v \bigr ) \le 1 + d_{T} \bigl (\alpha(0), \alpha(1) \bigr ) \\ & \le 5 + d_{T}\bigl (\pi(x), \pi(x') \bigr ) \\
    &\le \lambda d_{T} \bigl (\pi(x), \pi(x') \bigr ) + \lambda \\
    &= \lambda d_{W} \bigl (\pi_{W}(x), \pi_{W'}(x') \bigr ) + \lambda = N
\end{align*}
The lemma is proved.
\end{proof}

 \subsubsection{Bounded Geodesic Image}
 \begin{lemma}
 \label{lem:boundedgeo}
 $(X, \Lambda)$ satisfies Bounded Geodesic Image Axiom.
 \end{lemma}
 
 \begin{proof}
 We endow each $H_v$ with the HHS structure originating from the fact that $H_v$ is hyperbolic relative to the collection of boundary lines $\mathbb{L}_v$ (see Lemma~\ref{numberlines}). As a result, Bounded Geodesic Image Axiom is applied to this HHS structure. Since there are finitely many vertices up to $G$--action, there exists a constant $E >0$ such that properties in   Bounded Geodesic Image Axiom hold for every vertices $v \in T$. Let $r >0$ be the constant given by Remark~\ref{remark:useful}. We could enlarge $E$ so that $E > r$.

 We are going to verify \eqref{eqn:4}  in the Definition~\ref{defn:HHS}. Let $\gamma$ be a geodesic in $\calC W$.
Let $\gamma_{-}$ and $\gamma_{+}$ be the initial and terminal points of $\gamma$.
We consider the following cases:

{\it Case~1:} Suppose that $W = T$. Since $V$ is properly nested in $W$, it follows that $V$ is either  $\mathcal{R}_{v}$ or $\widehat{H}_v$ for some vertex $v$ in $T$. There are two subcases we will consider: $d(v, \gamma) \ge 2$ or $d(v, \gamma) \le 1$.
If $d(v, \gamma) \ge 2$ then the two geodesics $[\gamma_{-}, v]$ and $[\gamma_{+}, v]$ share the last common two consecutive edges, denoted by $e_1, e_2$. When $V = \mathcal{R}_v$ then it follows from the definition of $\rho^{W}_{V}$ that $\rho^{W}_{V}(x) = \mathcal{S}_{e_1,e_2} \cap F_{e_2} \subset \mathcal{R}_v$ for any $x \in \gamma$. Since $\mathcal{S}_{e_1,e_2} \cap F_{e_2}$ is a bounded subset in $\mathcal{R}_v$ with the diameter no more than $r$ (see Remark~\ref{remark:useful}), it follows that $\diam_{\calC V}(\rho^{W}_{V}(\gamma)) \le E$. If $V = \widehat{H}_v$ then  $\rho^{W}_{V}(x)$ is the apex-point $c_{e_2}$ for any $x \in \gamma$. Hence \eqref{eqn:4} is verified since $\diam_{\calC V}(\rho^{W}_{V}(\gamma)) \le E$.
If $d(v, \gamma) \le 1$ then $\gamma \cap \mathcal{N}_{2}(\rho^{V}_{W}) = \gamma \cap \mathcal{N}_{2}(v) \neq \emptyset$. Thus \eqref{eqn:4} is verified.

{\it Case~2:}
Suppose that $W = \mathcal{R}_v$ for some vertex $v \in T$. Then there is nothing to show since there is no element $V \in \Lambda$ such that $V$ is properly nested in $W$.

{\it Case~3:}
Suppose that $W = \widehat{H}_v$. Since $V$ is properly nested in $W$, it follows that $V = \mathcal{R}_w$ where $w$ and $v$ are adjacent vertices. Let $e$ denote the edge $[v,w]$, and let $\ell$ be the boundary line in $H_v$ associated to the edge $e$. Since \eqref{eqn:4} holds for hierarchically hyperbolic structure of $H_v$, we have that
\[
\diam_{\ell}(\rho^{\widehat{H}_v}_{\ell}(\gamma)) \le E \,\,\text{or}\,\, \gamma \cap \mathcal{N}_{E}(\rho^{\ell}_{\widehat{H}_v}) \neq \emptyset 
\] 
Note that $\diam_{\ell}(\rho^{\widehat{H}_v}_{\ell}(\gamma))$ and $\diam_{\mathcal{R}_w}(\rho^{\widehat{H}_v}_{\ell}(\gamma))$ are within a uniform bounded distance, and $\rho^{\ell}_{\widehat{H}_v}$ and $(\rho^{\mathcal{R}_W}_{\widehat{H}_v}$ are the apex point $c_e$. Thus it follows that there exists a constant $E'$ which does not depend on $v$, $\ell$, and $\gamma$ such that
\[
\diam_{\mathcal{R}_w}(\rho^{\widehat{H}_v}_{\ell}(\gamma)) \le E' \,\,\text{or}\,\, \gamma \cap \mathcal{N}_{E'}(\rho^{\mathcal{R}_W}_{\widehat{H}_v}) \neq \emptyset 
\] 
Replacing $E$ by $\max \{E, E' \}$ if necessary, we have that \eqref{eqn:4} holds. 
\end{proof}

\subsubsection{Partial Realization}
\begin{lemma}
$(X, \Lambda)$ satisfies Partial Realization Axiom.
\end{lemma}
\label{lem:partial}

\begin{proof}
Since $\{V_i\}$ be a family of pairwise orthogonal elements of $\Lambda$. It follows from the definition of orthogonality (see Section~\ref{orthogonality}) that the collection $\{V_i\}$ is either $\{\mathcal{R}_v, \widehat{H}_v \}$ for some vertex $v \in T$ or $\{V_i\} = \{\mathcal{R}_v, \mathcal{R}_w \}$ for some adjacent vertices $v$ and $w$ in $T$. In other words, the collection $\{V_i\}$ consists only two elements $\{V_1, V_2\}$.

Let $p_1$ and $p_2$ be two arbitrary elements in $\pi_{V_1}(X)$ and $\pi_{V_2}(X)$. We are going to define a point $x \in X$ so that it satisfies the conditions in Partial Realization Axiom as the following.

{\it Case~1:} Assume that $V_1 = \mathcal{R}_v$ and $V_2 = \widehat{H}_v$. Let $r >0$ be a sufficiently large constant such that it applies to both Remark~\ref{remark:useful} and Remark~\ref{rem:rdistance}. We note that $p_2$ is a point in $H_v$. By Remark~\ref{remark:useful}, there exists a point $\bar{x}$ in a boundary line $\ell$ of $H_v$ such that $d_{H_v}(p_2, \bar{x}) \le r$. Let $x$ be a point in $B_v = H_v \times \mathbb{R}$ such that its projection into $H_v$ is $\bar{x}$ and its $\mathbb R$--coordinate is the same as $p_1$.

{\it Case~2:} Assume that $V_1 = \mathcal{R}_v$ and $V_2 = \widehat{H}_w$ where $v$ and $w$ are adjacent vertices. Let $e$ denote the edge $[v,w]$. Since $p_1 \in \pi_{\mathcal{R}_v}(X)$ and $p_2 \in \pi_{\mathcal{R}_w}(X)$, it follows that they belong to boundary planes of $B_v$ and $B_w$ (this follows from the definition of $\pi_{\mathcal{R}_v}$ and $\pi_{\mathcal{R}_w}$). Choose $q_1$ and $q_2$ to be points in the plane $F_e$ such that the $\mathbb{R}$--coordinate of $q_1$ in $B_v = H_v \times \mathbb{R}$ is the same as of $p_1$ and the $\mathbb{R}$--coordinate of $q_2$ in $B_w = H_{w} \times \mathbb{R}$ is the same as of $p_2$. Let $\ell$ be the line in the plane $F_e$ that is parallel to $\mathbb R$--factor of $B_v$ and passes through $q_1$. Let $\ell'$ be the line in the plane $F_e$ that is parallel to $\mathbb R$--factor of $B_w$ and passes through $q_2$. These two lines $\ell$ and $\ell'$ intersect at a point, and we will denote this point by $x$.

Verifying the point $x$ satisfies the conditions in Partial Realization uses similar arguments as in Section~\ref{Transversality and consistency} and is hence omitted to avoid redundancy.
\end{proof}

\subsubsection{Uniqueness}
\begin{lemma}
\label{lem:uniqueaxiom}
$(X, \Lambda)$ satisfies Uniqueness Axiom.
\end{lemma}

\begin{proof}
We are going to verify the following statement.
For each $k \ge 0$, there exists $\Theta = \Theta(k)$ such that for any $x, y$ in $X$, $d(x, y) \ge \Theta$ then there exists $V \in \Lambda$ such that $d_{V}(x,y) \ge k$.

We endow $H_v$ with the HHS structure originating from the fact that $H_v$ is hyperbolic relative to the collection of boundary lines $\mathbb{L}_v$ (see Lemma~\ref{numberlines}). As a result, Uniqueness Axiom is applied to this HHS structure. Since there are finitely many vertices up to $G$--action, there exists a constant $\xi(k) >0$ such that for each vertex $v$, for any two points $\bar x$ and $\bar y$ in $H_v$ with $d_{H_v}(\bar x, \bar  y) \ge \xi(k)$ then either $d_{\widehat{H}_{v}}(\bar x, \bar y) \ge k$ or $d_{\ell}(\bar x, \bar y) \ge k$ for some boundary line $\ell \in \mathbb{L}_v$.

Let $\mu > 1$ be the constant given by Lemma~4.6 in \cite{NY20}. Let $\Theta$ be a sufficiently large constant such that 
\[
\Theta \ge \mu \bigl ( \mu + k (4 \sqrt{2}(k+3)) \bigr ) + 4 \sqrt{2} \xi(k)
\]
Let $\pi \colon X \to T$ be the index map given by Remark~\ref{rem:indexfunction}.

{\it Case~1:} Two vertices $\pi(x)$ and $\pi(y)$ has distance at least $k$ in $T$. 
Recall that we define the projection $\pi_{T} = \pi$. In this case, we will choose $V$ to be $T$. It follows that  $d_{V} (x, y) = d_{T}(x, y) := d_{T}(\pi_T (x), \pi_T (y)) = d_{T}(\pi(x), \pi(y)) \ge k$.

{\it Case~2:}  Two vertices $\pi(x)$ and $\pi(y)$ has distance at most $k$ in $T$.

Suppose that $\pi(x) = \pi(y)$. Let $v$ be the vertex $\pi(x) = \pi(y)$. Then we have that $x \in Y_{\pi(x)} = Y_v$ and  $y \in Y_{\pi(y)} = Y_v$. Let $\gamma$ be a geodesic in $B_v = H_v \times \mathbb{R}$ connecting $x$ to $y$. Let $\gamma_1$ and $\gamma_2$ be the projections of $\gamma$ into factors $H_v$ and $\mathbb R$ of $B_v$ respectively.  Since $d(x, y) \ge 
\Theta$, it follows that 
\[
\Len(\gamma_1) + \Len (\gamma_2) \ge \frac{\Theta}{\sqrt{2}}
\]
Hence either $\Len(\gamma_1)$ or $\Len (\gamma_2)$ is greater than or equal to $\frac{\Theta}{4\sqrt{2}}$. If $\Len (\gamma_2) \ge \frac{\Theta}{4\sqrt{2}}$ then we will choose $V = \mathcal{R}_v$. If \[ \Len(\gamma_1) \ge \frac{\Theta}{4\sqrt{2}} \ge \xi(k),\] i.e, $d_{H_v}((\gamma_{1})_{-}, (\gamma_{1})_{+}) \ge \xi(k)$, then either $d_{\widehat{H}_{v}}(\bar x, \bar y) \ge k$ or $d_{\ell}(\bar x, \bar y) \ge k$ for some boundary line $\ell \in \mathbb{L}_v$. In the first case, we will let $V = \widehat{H}_v$. In the later case, let $e$ be the edge in $T$ with $(e)_{-} = v$ such that $\ell$ is associated to $e$. We then let $V = \mathcal{R}_w$ where $w = (e)_{+}$.


Now we assume that $\pi(x) \neq \pi(y)$.  Let $e_{1} \cdots e_{n}$ be the geodesic edge path connecting  $\pi(x)$ to $\pi(y)$ and let  $p_i=\mathcal S_{e_{i-1}e_i}\cap \mathcal S_{e_{i}e_{i+1}}$ be the intersection point of adjacent strips, where $e_{0}:=x$ and $e_{n+1}:=y$. Hence we have the following path which is a concatenation of geodesics: $$\gamma := [p_0, p_{1}][p_{1}, p_{2}]\cdots [p_{{n-1}}, p_{n}][p_{n}, p_{n+1}]$$ where $p_0: =x$ and $p_{n+1} : = y$. Since we assume that $d_{T}(\pi(x), \pi(y)) \le k$, it follows that $n \le  k+2$.

By Lemma~4.6 in \cite{NY20}, the path $\gamma$ is  a $(\mu, \mu)$--quasi-geodesic. It follows that 
$
d_{X}(x,y) / \mu - \mu \le \Len(\gamma)
$, and  hence $\Theta / \mu - \mu \le  \Len(\gamma) = \sum_{i=0}^{n} \Len ([p_i, p_{i+1}])$ as $d(x, y) \ge \Theta$. It implies that there exists $i_{0} \in \{0, \cdots, n \}$ such that
\[
\frac{\Theta / \mu - \mu}{2(n+1)} \le \Len ([p_{i_0}, p_{ {i_0} +1}])
\]
Let $\alpha$ and $\beta$ be the projections of $[p_{i_0}, p_{ {i_0} +1}]$ into the factors $H_{v_{i_0}}$ and $\mathbb{R}$ of $B_{v_{i_0}}$ respectively (here $v_{i_0} : = e_{i_0} 
\cap e_{i_{0}+1}$). The above inequality implies
\[
\frac{\Theta / \mu - \mu}{2 \sqrt{2}(n+1)} \le \Len(\alpha) + \Len(\beta)
\] Thus either $\Len(\alpha)$ or $\Len (\beta)$ is greater than or equal to $\frac{\Theta / \mu - \mu}{2 \sqrt{2}(n+1)}$. We choose $V$ similarly as above (case $\pi(x) = \pi(y)$).
\end{proof}

\begin{proof}[Proof of Proposition~\ref{numberlines}]
The proof is the combination of the previous lemmas in this section.
\end{proof}

\section{Appendix: Poisson boundaries of hierarchically hype
rbolic spaces}
\label{appendix}
Let $G$ be a group that acts geometrically on a hierarchically hyperbolic space
(Definition \ref{defn:HHS}). Such a group we refer to as an HHS group, and note that conjecturally it is not equivalent to being a hierarchically hyperbolic group. In this section, we show that for an appropriate choice of $\kappa$ that only depends on the HHS group $G$, the 
$\kappa$-Morse boundary of $G$ serves as  a topological model for 
the Poisson boundaries of the pair $(G, \mu)$, where $\mu$ is a finitely-supported  
, non-elementary, generating measure.

We need to show that a generic sample path of such a random walk sublinearly tracks a $\kappa$-Morse quasi-geodesic ray. We will do so by showing that, in fact, 
the limiting quasi-geodesic ray is $\kappa$-weakly contracting((Definition 5.3 in \cite{QRT20}). This proof follows Section 8 of \cite{QRT20} where the same result is claimed for mapping class groups.

Without loss of generality, equip $X$ with a base-point $\go$ and we use $\go_S$ to denote a point in $\pi_S(\go)$. Fix $\go_S$ once and for all. 

\medskip

\subsection*{The hierarchy of geodesics} Let $S$ denote the maximal element in $
\Lambda$ and recall $\calC Y$ denoted the associated hyperbolic spaces for each $Y \in \Lambda$.
To every pair of points $x, y \in G$ one can associate a \emph{hierarchy
of geodesics}, which are a set of geodesic segments in each $\calC Y$ that connects 
 $\pi_Y(x)$  to 
$\pi_Y(y)$. Hence we also write $H(x, y) = \{ [x,y]_{Y} \}$ for specific given $Y \in \Lambda$. 
Given $H(x,y) = \{ [x,y]_{Y} \,| Y \in \Lambda \}$,  a realization of a hierarchy $H(x,y)$ is a uniform quasi-geodesic segment $\res(x, y)$ in 
$G$ connecting $x$ to $y$ where, for any element of $Y \in \Lambda$, the projection of $\res(x, y)$ to
$\calC Y$ is contained in a uniformly bounded neighborhood of the geodesic segment 
$[x,y]_Y$. The set of all realizations we denote $\res(x,y)$.

We can also replace $x$ or $y$ with elements of $G$ whose projection to $\calC S$ are infinite quasi-geodesic rays.

[We start with a (tight) geodesic $[x, \xi)_S$ in $\calC S$ and build $H(x, \xi)$, the same as before replacing, for every element $Y \in \Lambda$, $\pi_Y(y)$ with 
$\pi_Y(\xi)$. The realization $\res(x,\xi)$ of $H(x, \xi)$ is then a uniform quasi-geodesic 
in $G$ starting from $x$ such that the \emph{shadow} of $\res(x,\xi)$ in $\calC S$  converges to $\xi$. ]

\begin{definition}[Centers]
Since elements of $\Lambda$ are uniformly hyperbolic spaces, let them all be $D$--hyperbolic. There exists a constant $\delta(D)$ such that  for any three points $x, y, z$, the intersection of the $\delta(D)$--neighbourhood of the geodesic segments $[x, y], [y, z]$ and $[x, z]$ are non-empty. Given any $y \in \Lambda$, we use $ctr_Y (x,y,z)$ to denote any point in the said intersection. It follows from the construction of HHS that there exists a constant $D'$ such that: Given $x, y, z \in G$, for any $Y$ 
there exists a point $\eta \in G$ such that, for any $Y \subseteq S$, 
we have 
\begin{equation}\label{centerconstant}
d_Y\big(\eta_Y, \cent_Y(x,y,z)\big) \leq D'. 
\end{equation}

We call $\eta$ the \emph{center} of $x$, $y$ and $z$ and we denote 
it by $\cent(x, y, z)$. 
\end{definition}

We use the hierarchy paths to show: 

\begin{proposition} \label{P:hierarchy}
Let $p$ be the complexity of the hierarchy of $G$. For any $x, y \in X$, assume that 
$d_Y(x, y) \leq E$ for all $Y \neq S, Y \in \res$ and some $E > 1$. Then we have 
$$d_G(x, y) \prec d_S(x, y) \cdot E^p.$$
\end{proposition}

\begin{proof}
In view of the Distance Formula, we need to show 
\[
|H(x, y)| \prec d_S(x, y) \cdot E^p. 
\]
The restriction of $H(x,y)$ to a element of $\Lambda$ is again a hierarchy which we denote 
with $H_Y(x,y)$. Let $\alpha$ denotes minimal elements in the poset $\Lambda$, and let 
We check the statement of Proposition~\ref{P:hierarchy} inductively. When $p=1$ the statement is vacuously true. 
Now let $p \geq 2$. By induction, that for every element of $\Lambda$, the hierarchy $H_Y(x,y)$ satisfies $|H_Y(x,y)| \prec d_Y(x, y) \cdot E^{p-1}$. We have
\begin{align*}
|H(x,y)| &\prec \sum_{\alpha \in [x,y]_S} \left( \big| [x, y]_\alpha \big| 
 + \sum_{\alpha \sqsubset Y \sqsubset S} |H_Y(x,y)| \right)\\
&\prec \big| [x,y]_S \big| \cdot (E + 2 d_Y(x, y) \cdot E^{p-1}). 
\end{align*}
But $d_Y(x, y) \leq E$ and $\big|[x,y]_S\big| \prec d_S(x, y)$, thus by the Distance Formula $|H(x, y)| \prec d_S(x, y) \cdot E^p$. 
\end{proof}

Note that, once again we can replace each of $x,y,z$ with an element $\xi \in \partial \calC S$.  That is, $\cent(x,y, \xi)$ is a well-defined element of $G$. 
From now on, we will denote as $\go$ the identity element in $G$, which will 
function as base point.

\begin{definition}(Projections in HHS)
Let $D$ be given from \ref{centerconstant}, and let $\xi$ be an element of $G$ whose shadow in $\calC S$ is an infinite diameter, quasi-geodesic ray in $\partial \calC S$. 
We define a \emph{$D$-cloud of a ray in the direction of $\xi$} to be
\[ 
\calZ(\go, \xi)  := \Big\{ z \in G \ \big| \  d_{\calC Y} \big( z_Y,  [\go, \xi)_Y\big) \leq D \quad  \forall \, Y \Big\}. 
\]
\end{definition}
By construction, the realization $\res(\go, \xi)$ of the hierarchy $H(\go, \xi)$
is contained in $\calZ(\go, \xi)$. 
Fixing $\xi \in  \partial \calC S$, we define a projection map to the cloud:
\begin{align*}
\Pi_\xi \from G \to \calZ(\go, \xi)   \qquad\text{where}\qquad   \Pi_\xi(x) := \cent(\go, x, \xi), \qquad x \in G. 
\end{align*}

We now check that $\Pi_\xi$ is a $\kappa$-projection 
according to Definition \ref{weakprojection}.

\begin{lemma}
Let $\xi \in G$ be an element whose shadow in $\calC S$ is an infinite diameter, quasi-geodesic ray in $\partial \calC S$. , the map $\Pi_\xi$ is coarsely Lipschitz with respect to $d_G$. 
Furthermore, if 
$x \in \calZ(\go, \xi)$, then $d_G(x, \Pi_\xi(x))$ is uniformly bounded. As a consequence, $\Pi_\xi$ is a $\kappa$-projection. 
\end{lemma}

\begin{proof}
Consider points $x, x' \in G$ where $d_{w}(x, x') \leq 1$.
Then $x(\theta)$ and $x'(\theta)$ have a uniformly bounded intersection 
number, which implies that there exists a uniform constant $C_1>0$ such that
\[
\forall \, Y, \qquad  d_Y(x_Y, x'_Y) \leq C_1. 
\]
Let $\eta := \cent(\go, x, \xi)$ and $\eta' := \cent(\go, x', \xi)$. 
Since $\calC Y$ is hyperbolic, the dependence of $\eta_Y$ on $x_Y$ is Lipschitz, 
that is, there exists a uniform constant $C_2>0$ such that 
\[
 \forall \, Y \subseteq S, \qquad  d_Y (\eta_Y, \eta'_Y) \leq C_2. 
\]
Now, Proposition \ref{P:hierarchy} implies that 
\[
d_{w} (\Pi_\xi(x), \Pi_\xi(x')) \prec (C_2)^{p+1}
\]
which means $\Pi_\xi$ is coarsely Lipschitz. Similarly, if $x \in \calZ(\go, \xi)$
then for $\eta = \cent(\go, x, \xi)$ we have $ d_Y (x_Y, \eta_Y) \leq C_2$
for all elements of $\Lambda$ and hence, $d_Y(x, \Pi_\xi(x))\prec  (C_2)^{p+1}$. 
\end{proof}

\begin{lemma}\label{L:P-bound}
There exists $L > 0$ such that the following holds. Let $x, y \in G$, and let $\gamma \in G$ be an infinite geodesic ray based at $\go$,  and let $x_\gamma := c_{\gamma, y}(x)$.  For any $Y \in \Lambda$, if $d_Y(x, x_\gamma) \geq L$, we have 
$$d_Y(x,x_\gamma) \leq d_Y(x, y) +L.$$
\end{lemma}

\begin{proof}
Let $y_\gamma := \pi_\gamma(y)$. Since $Y$ is hyperbolic, there exists $L_1, R_0$ such that, if $d_Y(\pi_\gamma(x), \pi_\gamma(y)) \geq L_1$, the geodesic $\gamma_1 := [x, y]$ in $G$ and the broken geodesic  
\[
\gamma_2 
  := [x, \pi_\gamma(x)] \cup [\pi_\gamma(x), \pi_\gamma(y)] \cup [\pi_\gamma(y), y]
\] 
lie in a $R_0$-neighborhood of each other for the metric $d_Y$. 
Let $p_1, p_2$ be nearest point projections (in $G$), respectively, of $x_\gamma$ and $y_\gamma$ onto $[x, y]$.
This implies 
\begin{align}\label{deltathin}
d_Y(x, y) & \geq d_Y(x, p_1)  \geq d_Y(x, x_\gamma) - R_0 
\end{align}
which proves the claim, as long as $L \geq R_0$.

\end{proof}

\begin{proposition} \label{P:bounded-proj}
Let $\gamma$ be a hierarchy path that is a geodesic ray with $\kappa$-excursion. Then, the map 
$\Pi_\gamma$ defined above is a $\kappa$-projection map. 
Furthermore, there exist $D_{1} < 1, D_{2}>1$ such that for any two points 
$x, y \in G$ we have
\[
d_{G}(x, y) \leq D_{1} \cdot d_{G}(x, \gamma) \qquad \Longrightarrow \qquad d_S(\pi_{\gamma}(x), \pi_{\gamma}(y)) \leq D_{2}.
\] 
\end{proposition}

\begin{proof}
We now fix $L$ as given by Lemma \ref{L:P-bound}, and we start by contradiction, by assuming that $d_S(\pi_\gamma(x), \pi_\gamma(y)) \geq L$.

By Lemma \ref{L:P-bound}, $d_S(x, x_\gamma) \prec d_S(x,y)$, and $d_Y(x, x_\gamma) \prec d_Y(x, y)$ 
whenever $d_Y(x, x_\gamma)$ is large enough.
Now, applying the Distance Formula (Theorem 4.5 \cite{BHS19}) to the pair of points $(x, y)$ we have:
\begin{align*}
d_{G}(x, y) &\asymp \sum_{Y \in \Lambda} \lfloor d_{ Y}(x, y) \rfloor_{L} + d_S(x, y)\\
	& \succ \sum_{Y \in \Lambda} \lfloor d_{ Y}(x, x_\gamma) \rfloor_{L} + d_S(x, x_\gamma) - O(\delta) \qquad \textup{ by Eq.}~\eqref{deltathin} \\
	& \asymp d_{G}(x, x_\gamma). 
\end{align*}
That is to say, there exists $D_{1} = D_{1}(L, \delta) $ such that 
\[
d_{G}(x, y) \geq D_{1} \cdot d_{G}(x, x_\gamma),
\]
which is a contradiction since $x_\gamma \in \gamma$. Therefore, setting $D_2 = L$ yields
\begin{equation*}
d_{G}(x, y) \leq D_{1} \cdot d_{G}(x, \gamma) \qquad  \Longrightarrow \qquad d_S((\pi_\gamma(x), \pi_\gamma(y)) \leq D_{2}.
\qedhere
\end{equation*}

\end{proof}

\subsection*{Logarithmic projections}

We now consider the set of elements in $G$ whose projection to $\calC S$ are quasi-geodesic rays in $\partial \calC S$ that have \emph{logarithmically bounded 
projection} to all elements of $\Lambda$. Recall in Definition~\ref{defn:HHS}, When $V \sqsubseteq W$ then there is a specific subset $\rho_{W}^{V} \subset \calC W$ such that $\diam (\rho_{W}^{V}) \le \xi$.

Given a non-maximal element $Y \in \Lambda$, let

\[
\Norm{Y}_S :=d_{S}(\go_S, \rho^Y_S). 
\] 
In comparison, for $x \in G$, define
\[
\Norm{x}_S :=d_{S}\big(\go_S, \pi_S(x)) \big).  
\]

\begin{definition}
For a constant $c > 0$, let $\calL$ be the set of elements in $G$ whose shadows to an infinite quasi-geodesic ray $\xi \in  \partial \calC S$ such that 
\begin{equation} \label{E:log-proj}
 d_Y(\go, \xi) \leq c \cdot  \log \Norm{Y}_S
\end{equation}
for every non-maximal element $Y \in \Lambda$.
\end{definition}

\begin{proposition} \label{P:contracting-morse}
For any $\xi \in \calL$, the set $\calZ(\go, \xi)$ is $\kappa$-weakly contracting, where 
$\kappa(r) = \log^{p}(r)$. Furthermore, any realization $\res(\go, \xi)$ of
the hierarchy $H(\go, \xi)$ is also $\kappa$-Morse.
\end{proposition}

\begin{proof}
In this proof, we use the notations $\prec_c$ and $O_c$ to mean that the implicit constants additionally depend on $c$. 
Let $\calZ= \calZ(\go, \xi)$. Given $x, x' \in X$ where  
\[
D_1 \cdot d_G(x, x') < d_G(x, \calZ),
\]
let $y = \Pi_\xi(x)$, $y' = \Pi_\xi(x')$. We claim that, for every proper element of $\Lambda$,
\[
d_Y(y, y') \prec_c \log \Norm{x}_S. 
\]

Since $\calC S$ is hyperbolic, nearest point projection in $\calC S$ is coarsely distance decreasing, hence $\Norm{y}_S \prec \Norm{x}_S$. Also, by 
Theorem \ref{P:bounded-proj} 
\begin{equation} \label{eq:DR}
d_{S}(y, y') \leq D_2
\end{equation}
therefore, $\Norm{y'}_S \prec \Norm{x}_S$. Which means, for every curve $\alpha$ in the 
geodesic segment $[y, y']_S$ in $\calC S$ we have 
$d_S(\go, \alpha) \prec \Norm{x}_S$. Axiom (7) (Bounded Geodesic Image)
implies that if $d_Y(y, y')$ is large then $d_S([y, y']_S, \partial Y) \prec 1$, hence 
\[
\Norm{Y}_S \prec \Norm{x}_S. 
\]

By the definition of $\Pi_\xi$, $y_Y$ and $y'_Y$ are $D$-close to the geodesic
segment $[\go, \xi]_Y$ in $\calC Y$ and, by assumption, the length of this 
segment is at most a uniform multiple of $\log \Norm{Y}_S$. 
Therefore, 
\[
d_Y(y, y') \prec \big| [\go, \xi]_Y\big| \prec_c \log \Norm{Y}_S \prec \log \Norm{x}_S. 
\]
In view of \eqnref{eq:DR} and Proposition \ref{P:hierarchy}, we get 
\[
d_G(y,y') \prec_c  \log^p \Norm{x}_S. 
\]
Now, by Theorem A.1 in \cite{QRT20}, $\calZ(\go, \xi)$ is 
$\kappa$-Morse. Let $m_\calZ$ be the associated Morse gauge for $\calZ(\go, \xi)$. 

Now we show $\res(\go, \xi)$ is also $\kappa$-Morse. Assume $\kappa'$ and 
$r>0$ be given (see \defref{D:k-morse}) and, using the fact that 
$\calZ(\go, \xi)$ is $\kappa$-Morse, let $R$ be a radius such that, for any 
$(q,Q)$-quasi-geodesic ray 
$\beta$ in $G$ with $m_\calZ(q,Q)$ small compared to $r$, we have 
\[
d_G(\beta_R, \calZ(\go, \xi) ) \leq \kappa'(R) 
\qquad \Longrightarrow\qquad
\beta|_r \subset \calN_\kappa (\calZ(\go, \xi), m_\calZ(q,Q)). 
\]
Also, assume 
\[
d_G(\beta_R, \res(\go, \xi) ) \leq \kappa'(R). 
\]
We need to show that every $x \in \beta|_r$ is close to $\res(\go, \xi)$. 

Since $\res(\go, \xi) \subset \calZ(\go, \xi)$ we can still conclude that 
there is a point $y \in \calZ(\go, \xi)$ with
\[
d_G(x,y) \leq m_\calZ(q, Q) \cdot \kappa(x). 
\]
In fact $y$ can be taken to be $\Pi_\xi(x)$ and hence $\Norm{y}_S \prec \Norm{x}_S$. 
Let $z$ be a point in $\res(\go, \xi)$ where $d_S(z_S , y_S) \leq D$
(such a point exists since the shadow of $\res(\go,\xi)$ to $\calC S$ is the
geodesic ray $[\go, \xi)_S$). Since $y,z \in \calZ(\go, \xi)$, we have for every 
element of $\Lambda$ that 
\[
d_Y(y,z) \prec_c \log \max(\Norm{y}_S, \Norm{z}_S) \prec \log (\Norm{x}_S + D) 
 \prec \log \Norm{x}_S.
\]
Therefore, by \propref{P:hierarchy}, we have 
\[
d_G(y,z) \prec_c  \log^p \Norm{x}_S \prec \kappa(x).
\]
And hence, 
\[
d_G(x,z) \leq d_G(x,y) + d_G(y,z) \prec_c m_\calZ(q,Q) \cdot \kappa(x). 
\]
We have shown 
\[
\beta|_r \subset \calN_\kappa \Big(\res(\go, \xi), O_c\big(m_\calZ(q,Q)\big)\Big). 
\]
That is, $\res(\go, \xi)$ is $\kappa$-Morse with a Morse gauge 
$m_\res = O_c(m_\calZ)$. 
\end{proof}

\subsection{Convergence to the $\kappa$-Morse boundary}

Let $\mu$ be a probability measure on $G$. We say that $\mu$ is \emph{non-elementary} if the semigroup generated by its support 
contains two loxodromic elements with disjoint fixed sets in $ \partial \calC S$.

Let us recall some useful facts on random walks on any HHS group.

\begin{theorem} \label{T:RW}
Let $\mu$ be a finitely supported, non-elementary probability measure on an HHS group $G$. Then: 
\begin{enumerate}
\item
For almost every sample path $\omega = (w_n)$, the sequence $(w_n)_S$ converges to a point $\xi_\omega$ in the Gromov boundary of $\calC S$. 
\item
Moreover, there exists $l > 0, c < 1$ such that 
$$\mathbb{P}\left( d_S(\go, w_n) \geq  l n  \right) \geq 1 - c^n$$
for any $n$.
\item
Further, for any $k > 0$ there exists $C > 0$ such that 
$$\mathbb{P}\left( d_S(w_n, \gamma_\omega) \geq C \log n \right) \leq n^{-k}$$
for any $n$, where $\gamma_\omega = [\go, \xi_\omega)_S$. 
\end{enumerate}
\end{theorem}

Claim (2) is proven by Maher (\cite{Maher}, \cite{MaherExp}), while (1) and (3) are proven by Maher-Tiozzo in \cite{MaherTiozzo}. 
Also, exactly the same proof as in \cite[Theorem A.17]{QRT19} yields for HHS groups that for  finitely supported, non-elementary probability measure on $G$.
Then for any $k > 0$ there exists $C > 0$ such that for all $n$ we have
$$\mathbb{P}\left(\sup_Y d_{Y}(\go, w_n) \geq C \log n \right) \leq C n^{-k},$$
where the supremum is taken over all (proper) elements of $\Lambda$ of $S$. 
As a consequence, for almost every sample path there exists $C > 0$ such that for all $n$
$$\sup_Y d_{Y}(\go, w_n ) \leq C \log n.$$

It follows that almost every sample path converges to a point in the $\kappa$-Morse boundary of the HHS group, 
where $\kappa(r) = \log^p(r)$. We now complete the proof of Theorem \ref{T:PB-intro} by identifying the $\kappa$-Morse boundary with the Poisson boundary. 

\begin{theorem}

 \label{T:mcg-poiss}
Let $\mu$ be a non-elementary, finitely supported measure on an HHS group $G$. Then for $\kappa(r) := \log^p(r)$, where $p$ is the complexity of the hierarchy, the $\kappa$-Morse boundary is a topological model for the Poisson boundary of $(G, \mu)$. 
\end{theorem}

\begin{proof}
Since almost every sample path sublinearly tracks a $\kappa$-Morse geodesic ray, with $\kappa(r) = \log^p(r)$, 
so we can apply Theorem \ref{T:poiss-general}.
\end{proof}

\bibliographystyle{alpha}

\end{document}